\documentclass[12pt]{amsart} 
\usepackage[mathscr]{eucal} 
\usepackage{amsmath,amsfonts} 
\usepackage{amssymb}
\usepackage{color}
\usepackage{graphicx}
\usepackage{hyperref}
\parskip=\smallskipamount 
\newtheorem{theorem}{Theorem}[section] 
\newtheorem{proposition}[theorem]{Proposition} 
\newtheorem{corollary}[theorem]{Corollary} 
\newtheorem{lemma}[theorem]{Lemma} 
\theoremstyle{definition} 
\newtheorem{definition}[theorem]{Definition}

\newtheorem{remark}[theorem]{Remark}

\makeatletter 
\@addtoreset{equation}{section} 
\makeatother

\newcommand{\CC}{{\mathbb C}} 
\newcommand{\NN}{{\mathbb N}}

\newcommand{\RR}{{\mathbb R}} 
 
\newcommand{\cA}{{\mathcal A}} 
\newcommand{\cB}{{\mathcal B}} 
 
\newcommand{\cD}{{\mathcal D}} 
 
\newcommand{\cF}{{\mathcal F}} 
\newcommand{\cG}{{\mathcal G}} 
\newcommand{\cH}{{\mathcal H}} 
\newcommand{\cI}{{\mathcal I}}
\newcommand{\cJ}{{\mathcal J}} 
\newcommand{\cK}{{\mathcal K}} 
\newcommand{\cL}{{\mathcal L}} 
\newcommand{\cM}{{\mathcal M}}

\newcommand{\cR}{{\mathcal R}}

\newcommand{\cZ}{{\mathcal Z}} 

\newcommand{\bG}{\mathbf{G}}
\newcommand{\bH}{\mathbf{H}}

\newcommand{\fC}{\mathfrak{C}}

\newcommand{\ra}{\rightarrow} 
\newcommand{\ol}{\overline}

\let\phi=\varphi

\newcommand{\de}{\mathrm{d}}

\newcommand{\lin}{\operatorname{Lin}}
\newcommand{\supp}{\operatorname{supp}}
\newcommand{\Lloc}{\cB_{\mathrm{loc}}}

\newcommand{\esssup}{\operatorname{ess-sup}}
\newcommand{\esssupp}{\operatorname{ess-supp}}

\newcommand{\nr}[1]{\vspace{0.1ex}\noindent\hspace*{12mm}\llap{\textup{(#1)}}} 

\date{\today}
 
\begin{document} 
\title[Abelian Locally von Neumann Algebras]{Functional Models of \\[1ex]
Abelian Locally von Neumann Algebras and\\[1ex] Direct Integrals of Locally Hilbert Spaces}

\begin{abstract} We obtain a functional model for an arbitrary Abelian locally von Neumann algebra acting on a 
representing locally Hilbert space under the assumption that the index directed set is sequentially finite, hence 
countable, in terms of locally essentially bounded functions on 
strictly inductive systems of measure spaces, which can be viewed as the reduction theory of this kind of 
operator algebras. 
Then, we single out the concept of a direct integral of locally Hilbert spaces and 
the concepts of locally decomposable and locally diagonlisable operators and we show that these form locally 
von Neumann algebras that are commutant one to each other. Finally, we show that any Abelian locally von 
Neumann algebra which, in addition, acts on separable representing locally Hilbert spaces, is spatially 
isomorphic with the Abelian locally 
von Neumann algebra of all locally diagonlisable operators on a certain direct integral of 
locally Hilbert spaces with respect to a certain strictly inductive system of locally finite measure spaces on standard Borel spaces.
    
\end{abstract}
\author[A.~Gheondea]{Aurelian Gheondea}
\address{Institute of Mathematics of the Romanian Academy, Calea Grivi\c tei 21, Bucharest, Romania \emph{and} Bilkent University, Ankara, 6800 T\"urkiye}
\email{a.gheondea@imar.ro \emph{and} aurelian@fen.bilkent.edu.tr}
\author[C.J.~Kulkarni]{Chaitanya J. Kulkarni}
\address{Theoretical Statistics and Mathematics Unit, Indian Statistical Institute, Delhi Centre, 7 S. J.
S. Sansanwal Marg, New Delhi 110016, India}
\email{chaitanyakulkarni58@gmail.com}
\author[S.K.~Pamula]{Santhosh Kumar Pamula}
\address{Department of Mathematical Sciences, Indian Institute of Science Education and Research (IISER) Mohali, Knowledge City, S.A.S Nagar, Punjab 140306, Punjab, India}
\email{santhoshkp@iisermohali.ac.in}
\maketitle

\section*{Introduction}

Locally $C^*$-algebras, or pro $C^*$-algebras, or $LMC$-algebras, are projective limits of $C^*$-algebras 
in the category of 
locally convex $*$-algebras and have been considered and investigated 
by G.~Allan \cite{Allan}, C.~Apostol \cite{Apostol},  A.~Inoue in \cite{Inoue}, W.B.~Arveson \cite{Arveson},
K.~Schm\"udgen \cite{Schmudgen}, and  N.C.~Phillips \cite{Phillips}. These operator algebras were useful in 
connection with the theory of Hilbert modules over locally convex $*$-algebras initiated by the works of 
A.~Mallios \cite{Mallios}, in connection to dual algebraic structures related to free products by 
D.V.~Voiculescu \cite{Voiculescu}, and in connection to KK-theory by J.~Weidner \cite{Weidner}. 
With respect to operator
representations of locally $C^*$-algebras, locally Hilbert spaces have been considered by A.~Inoue \cite{Inoue} 
and, from a different point of view, operator theory in locally Hilbert spaces was considered by A.A.~Dosiev 
\cite{Dosi1}, who called
them quantised domains of Hilbert spaces, and by D.J.~Karia and Y.M.~Parmar in \cite{KariaParmar} and
A.~Gheondea \cite{Gheondea1}. The spectral theory of locally normal operators was performed by 
A.~Gheondea, \cite{Gheondea2}, \cite{Gheondea3}, and \cite{Gheondea4}, in particular, the three forms of the 
spectral theorem for locally 
normal operators have been obtained. On the one hand, locally normal operators generate Abelian locally 
von Neumann algebras, which are operator algebra types of locally $W^*$-algebras, see
M.~Fragoulopoulou, \cite{Fragoulopoulou} and M.~Joi\c ta, \cite{Joita1}, \cite{Joita2}, and also, A.A.~Dosi 
\cite{Dosi3} who studied Abelian $W^*$-algebras, hence a natural question is how these results fit into the 
theory of Abelian locally von Neumann algebras. On the other hand, in \cite{Gheondea4}, a connection of 
operator theory on locally Hilbert spaces with the analysis on fractal sets, cf.\ J.~Kigami \cite{Kigami}, differential 
equations on fractal sets, cf.\ 
R.S.~Strichartz \cite{Strichartz}, and the spectral theory on the Hata tree-like self-similar set, cf.\  M.~Hata 
\cite{Hata} and A.~Brzoska \cite{Brzoska}, was established and motivations for the technical assumptions that 
the index directed set is sequentially finite have been provided. Again, these raise the question on the structure 
of Abelian locally von Neumann algebras.

Locally $C^*$-algebras are necessary for certain applications, as substantiated by the works cited 
above, due to the fact that, when compared to $C^*$-algebras, they cover the case of unbounded operators. 
For example, the category of Abelian locally $C^*$-algebras is equivalent, by a generalisation of Gelfand's 
duality, 
with the opposite of the category of colimits of compact Hausdorff spaces in the category of Hausdorff spaces 
and continuous maps, see R.~El Harti and G.~Luk\'acs \cite{ElHartiLukacs}.
Locally von Neumann algebras are considered to be less useful, when compared to locally $C^*$-algebras, 
because of the existing concept of unbounded operators that are affiliated to von Neumann algebras, which 
leads to the impression that von Neumann algebras have an intrinsic power of handling unbounded operators. 
However, the structure of a locally von Neumann algebra is heavily influenced by the underlying index directed 
set, e.g.\ see examples 3.17 through 3.19 in \cite{Gheondea4}, hence its centre may be rather complicated,    
while the affiliated operators may not catch the whole characteristics. In addition, the 
suitable framework to deal with locally von Neumann algebras is that in which, the underlying locally Hilbert 
space is representing (commutative), which has the consequence that the 
locally von Neumann algebra has a rich centre. In this case, a reduction theory of locally von Neumann algebra 
becomes highly nontrivial and certain canonical forms are interesting objects. All these require a representation 
theory of Abelian locally von Neumann algebras with direct integrals of locally Hilbert spaces and then, a 
reduction theory of locally von Neumann algebras with respect to their centres. In view of the proposed approach 
to a spectral theory of differential operators on fractal sets, cf.\ \cite{Kigami}, \cite{Strichartz}, \cite{Brzoska}, a 
Laplace operator operator on the Hata tree-like self-similar set $X$
should be a selfadjoint operator in a certain locally von Neumann algebra on $X$, when viewed as 
the strictly inductive system of measure spaces associated to the tree structure of $X$, see Example~3.19 in 
\cite{Gheondea4}.

The aim of this article is to investigate Abelian locally von Neumann algebras from the points of view of 
functional models and 
their  representations as algebras of locally diagonalisable operators on certain direct integrals of locally Hilbert 
spaces. In a certain way, these can be viewed as generalisations of the functional models for locally normal 
operators obtained in \cite{Gheondea4}, a reduction theory for these operator algebras, 
and also as a counter-part of the representations of Abelian locally von Neumann algebras obtained in 
\cite{Dosi3}. Our approach is based, on the one hand, on known results on functional models of Abelian von 
Neumann algebras and their  representations on direct integrals of Hilbert spaces in the sense of J.~von 
Neumann \cite{vonNeumann} and, on the other hand, on strictly inductive systems of measure spaces used in 
\cite{Gheondea3} and \cite{Gheondea4}. In this enterprise, 
there are a few challenges that we had to face. First, the technical assumption on the underlying locally Hilbert 
space to be representing, commutative in the sense of \cite{Dosi1}, is necessary if we want to obtain spatial 
isomorphism. Second, for the general functional model, we had to impose the technical 
assumption on the index directed set 
to be sequentially finite, similar to \cite{Gheondea4}, which is justified by an induction process that 
we had to use. Third, we had to find the suitable 
definition of the direct integral of locally Hilbert spaces, which we did in 
Subsection~\ref{ss:bblhs}, through the concept of Borel bundles of locally Hilbert spaces, and the right concepts 
of locally 
decomposable and locally diagonalisable operators, which we did in Subsection~\ref{ss:ldldo}. 
In this respect, first a very helpful idea was the use of the Borel bundles of Hilbert spaces in 
the presentation of P.~Muhly \cite{Muhly}
and based on ideas of A.~Ramsey \cite{Ramsey} and second, the concept of strictly inductive 
systems of measure spaces allows us to calculate the completion of a direct integral of locally Hilbert spaces as 
an explicit direct integral of Hilbert space, see 
Theorem~\ref{t:dintcomp}, in the special case when the index set is a countable directed set. This technical 
condition does not 
provide new limitations of our results since, anyhow, the main results are obtained under the technical condition 
that the index directed set is sequentially finite and hence countable.

In addition to the obstructions made explicit above, there is one issue that requires special care, in 
connection with the requirement, in the definition of locally Hilbert spaces, that we have inclusions, and not 
simply embeddings, and, consequently, in connection with the definition of strictly inductive limits of measurable spaces which 
requires that we have inclusions and not simply embeddings. These make technical complications in view of the 
fact that, in the functional model, we have orthogonal 
sums of countable collections of locally Hilbert spaces and 
direct integrals of locally Hilbert spaces. To a certain extent, these complications are alleviated by the 
assumption on the index directed set to be sequentially finite, but more is needed. 
On the one hand, when orthogonal direct sums are used, we have to single out a special type of these 
orthogonal sums which makes them unique, and for this we use the idea of a reproducing kernel Hilbert space, 
see Subsection~\ref{ss:ods}. On the other hand, when considering locally Hilbert spaces produced by strictly 
inductive systems of measure spaces, we have to use functions defined on the whole inductive limit, see 
Subsection~\ref{ss:sisms}.

The main results of this article are presented in Section~\ref{s:dilhs}. First, in Subsecion~\ref{ss:bblhs} we get 
the definition of a Borel bundle of separable locally Hilbert spaces which provides direct integrals of locally 
Hilbert spaces, then in Proposition~\ref{p:dintrep} we show that these locally Hilbert spaces are always 
representing and in Theorem~\ref{t:dintcomp} we
calculate their completions to a direct integral Hilbert space, under the assumption that the index directed 
set is countable. In the next 
subsection, we define the concepts of locally decomposable and locally diagonalisable operators and in 
Theorem~\ref{t:locdecdeg} we show that these classes of locally bounded operators yield locally von Neumann 
algebras such 
that each of it is the commutant of the other. Under the assumption that the index ordered set is sequentially 
finite, we obtain in 
Theorem~\ref{t:avn1} the functional model of Abelian locally von Neumann algebras acting on representing 
locally Hilbert spaces
on pseudo-concrete locally Hilbert spaces, which can be viewed as the reduction for these kind of operator 
algebras and has the 
advantage that no separability condition is needed. Finally, in Theorem~\ref{t:avn2} we show that if, in addition, 
the locally Hilbert spaces 
are separable, we obtain that any Abelian locally von Neumann algebra is spatially isomorphic to the locally von 
Neumann algebra of locally diagonalisable operators on a direct integral locally Hilbert space, with respect to a 
strictly inductive system of locally finite measure spaces which are standard Borel spaces.

In this article we use many concepts and results that need a review and, because of that, in Section~\ref{s:npr} we recall the 
necessary concepts of locally Hilbert spaces, their orthogonal direct sums, locally $C^*$-algebras, locally bounded operators, 
locally von Neumann algebras, and strictly inductive systems of measure spaces, 
while in Section~\ref{s:rdihs} we review the concepts of Borel bundles of Hilbert spaces, direct integral of Hilbert spaces, and the 
von Neumann algebras of decomposable and diagonalisable operators, respectively.

\section{Notation and Preliminary Results on Operator Algebras on\\ Locally Hilbert Spaces} \label{s:npr}

In this section, we fix the notation and recall some basic definitions 
and results, mainly from \cite{Gheondea1}, \cite{Gheondea2}, \cite{Gheondea3}, and \cite{Gheondea4}, and a few others, in 
order to provide a setup for locally von Neumann algebras.

\subsection{Locally Hilbert Spaces.}\label{ss:lhs} 
A \emph{locally Hilbert space} is an \emph{inductive limit}
\begin{equation}\label{e:injlim} 
\cH=\varinjlim\limits_{\lambda\in\Lambda}\cH_\lambda=\bigcup_{\lambda\in\Lambda}
\cH_\lambda,
\end{equation} of a 
\emph{strictly inductive system of 
Hilbert spaces} $(\cH_\lambda)_{\lambda\in
\Lambda}$, that is,
\begin{itemize}
\item[(lhs1)] $(\Lambda,\leq)$ is a directed set; 
\item[(lhs2)] $(\cH_\lambda, \langle\cdot,\cdot\rangle_{\cH_\lambda})_{\lambda\in\Lambda}$ 
is a net of Hilbert spaces;
\item[(lhs3)] for each $\lambda,\mu\in\Lambda$ with $\lambda\leq \mu$  we have 
$\cH_\lambda\subseteq\cH_\mu$; 
\item[(lhs4)] for each $\lambda,\mu\in\Lambda$ with $\lambda\leq\mu$ the inclusion map 
$J_{\mu,\lambda}\colon \cH_\lambda\ra\cH_\mu$ is isometric, that is, 
\begin{equation}\langle x,y\rangle_{\cH_\lambda}=\langle x,y\rangle_{\cH_\mu},\mbox{ for all }
x,y\in\cH_\lambda.\end{equation}
\end{itemize} 
For each $\lambda\in\Lambda$, letting $J_\lambda\colon\cH_\lambda\ra\cH$ be the inclusion
of $\cH_\lambda$ into the vector space $\cH=\bigcup\limits_{\lambda\in\Lambda}\cH_\lambda$, the \emph{inductive limit
topology} on $\cH$ is the strongest one which makes the linear maps $J_\lambda$ 
continuous for all $\lambda\in\Lambda$. 

On the vector space $\cH$, a canonical inner product $\langle\cdot,\cdot\rangle_\cH$ can be defined as follows: 
\begin{equation}\label{e:lip}
\langle h,k\rangle_\cH=\langle h,k\rangle_{\cH_\lambda},\quad h,k\in\cH,
\end{equation} where $\lambda\in\Lambda$ is any index for which $h,k\in\cH_\lambda$. 
With notation as before, it
follows that the definition of the inner product as in \eqref{e:lip}
is correct and, for each $\lambda\in\Lambda$, 
the inclusion map $J_\lambda\colon(\cH_\lambda,\langle\cdot,\cdot\rangle_{\cH_\lambda})\ra
(\cH,\langle\cdot,\cdot\rangle_{\cH})$ is isometric. This implies that, letting $\|\cdot\|_\cH$ 
denote the norm induced by the inner product $\langle\cdot,\cdot\rangle_\cH$ on $\cH$, the
\emph{norm topology} on $\cH$ is weaker than the inductive limit topology of $\cH$. 
Since the norm topology is Hausdorff, it follows that the inductive limit topology on $\cH$ 
is Hausdorff as well. In the following, we let $\widetilde\cH$ denote a Hilbert space completion of the
inner product space $(\cH,\langle\cdot,\cdot\rangle).$ 

\begin{remark} Two important observations are in order. 

(a) First, a locally Hilbert space $\cH$ is not an inductive limit in the category of 
Hilbert spaces but in the more general category of locally convex spaces. However, its completion 
$\widetilde \cH$ is an inductive limit in the category of Hilbert spaces. 

(b) Second, there is the concept of \emph{quantised domain} of a 
Hilbert space, cf.\ \cite{Dosi1}, which is a triple $(\cG,\{\cD_p\}_{p\in S},\cD)$, where $\cG$ is a Hilbert space and 
$\{\cD_p\}_{p\in S}$ is a system of Hilbert subspaces of $\cG$, for some directed set $S$, 
subject to the conditions that, whenever $p,q\in S$ and $p\leq q$ it follows that $\cD_p\subseteq \cD_q$ and 
$\cD=\bigcup\limits_{p\in S}\cD_p$ is dense in $\cG$. Any locally Hilbert space
$\cH=\varinjlim\limits_{\lambda\in\Lambda}\cH_\lambda$ gives rise to a quantised domain 
$(\widetilde\cH,\{\cH_\lambda\}_{\lambda\in\Lambda},\cH)$. Conversely, for any quantised domain 
$(\cG,\{\cD_p\}_{\in S},\cD)$, the net $(\cD_p)_{p\in S}$ is a strictly inductive system of Hilbert spaces and its 
inductive limit $\cD:=\varinjlim\limits_{s\in S} \cD_p$ is a locally Hilbert space such that $\cD$ is dense in $\cG$.
However, the concept of quantised domain $(\cG,\{\cD_p\}_{\in S},\cD)$  of a Hilbert 
space is more restrictive because the emphasis is on the Hilbert space $\cG$, which is 
fixed, while in the concept of a locally Hilbert space $\cH=\varinjlim\limits_{\lambda\in\Lambda}\cH_\lambda$, the 
emphasis is on the strictly inductive system of Hilbert spaces $(\cH_\lambda)_{\lambda\in\Lambda}$ and its 
completion $\widetilde \cH$ is one of the many possible choices. This distinction is important for our approach.
\end{remark}
Given a locally Hilbert space $\cH=\varinjlim\limits_{\lambda\in\Lambda}\cH_\lambda$, for any $\lambda\in\Lambda$ 
there exists a unique projection $P_\lambda$ from $\cH$ to its subspace $\cH_\lambda$, which is 
Hermitian in the sense that $\langle P_\lambda h,k\rangle_{\cH}=\langle h,P_\lambda k\rangle_{\cH}$, for all 
$h,k\in \cH$. For example,\ see Lemma~3.1 in \cite{Gheondea1}.
\begin{definition}\label{d:rep}
A locally Hilbert space $\cH=\varinjlim\limits_{\lambda\in\Lambda}\cH_\lambda $ is called \emph{representing}
if, letting $P_\lambda\colon\cH\ra\cH$ denote the unique Hermitian projection of $\cH$ 
onto $\cH_\lambda$, for arbitrary 
$\lambda\in\Lambda$,  we have
\begin{itemize}
\item[(lch5)] $P_\lambda P_\nu=P_\nu P_\lambda$ for all $\lambda,\nu\in\Lambda$.
\end{itemize}
In this case, the strictly inductive system of Hilbert spaces $(\cH_\lambda)_{\lambda\in\Lambda}$ is called
\emph{representing} as well.
\end{definition}
A quantised domain $(\cG,\{\cD_p\}_{\in S},\cD)$ with the property that its underlying locally Hilbert space 
$\cD=\varinjlim\limits_{s\in S}\cD_p$ is representing is called \emph{commutative} in \cite{Dosi2}. For the proof of the following result see Lemma~3.2 in \cite{Gheondea4}.

\begin{lemma}\label{l:repof} A locally Hilbert space is representing if and only if, for any 
$\lambda,\nu\in\Lambda$ and any $\epsilon\in\Lambda$ such that $\lambda\leq\epsilon$ and 
$\nu\leq\epsilon$ (since $\Lambda$ is directed, we always can find such an $\epsilon$), letting 
$P_{\lambda,\epsilon},P_{\nu,\epsilon}\in\cB(\cH_\epsilon)$ 
be the orthogonal projections onto $\cH_\lambda$ and $\cH_\nu$, respectively,  we have
\begin{equation}\label{e:pale}
P_{\lambda,\epsilon}P_{\nu,\epsilon}=P_{\nu,\epsilon}P_{\lambda,\epsilon}.
\end{equation}
\end{lemma}

\subsection{Orthogonal Direct Sums of Locally Hilbert Spaces.}\label{ss:ods}
We need to perform certain operations with locally Hilbert spaces, such as orthogonal direct sums. This operation 
needs a careful treatment and, because of that, we first clarify the sense in which we have to consider 
orthogonal direct sums of Hilbert spaces, following a particular case of linearisations and reproducing kernel 
Hilbert spaces of operator valued kernels as in \cite{GheondeaTilki}. 
Let $\cJ$ be an arbitrary set and consider a bundle
of Hilbert spaces $\bH=\{\cH_j\}_{j\in \cJ}$. We consider 
\begin{equation*}\cJ\ast \bH := \bigsqcup_{j\in\cJ}\cH_j=\{(j,h)\mid j\in\cJ,\ h\in \cH_j\}
\end{equation*}
and the canonical map $\pi\colon \cJ\ast \bH\ra \cJ$, $\pi(j,h)=j$ for all $(j,h)\in\cJ\ast \bH$. A \emph{section} of
the bundle $\cJ\ast \bH$ is a right inverse of $\pi$, that is, $f\colon \cJ\ra \cJ\ast \bH$ subject to the condition that,
for any $j\in\cJ$ we have $f(j)=(j,h)$, with $h\in\cH_j$. In order to simplify the notation, by the identification of the fiber 
$\{(j,h)\mid h\in \cH_j\}$ with $\cH_j$, we can write simply that $f(j)\in\cH_j$ for all $j\in\cJ$.
Let $\cF_{\bH}(\cJ)$ denote the vector space of all sections of the bundlel $\cJ\ast\bH$. 

An \emph{orthogonal direct sum} of $\bH$ is, 
by definition, a pair $(\cH;\iota)$ subject to the following conditions.
\begin{itemize}
\item[(ds1)] $\cH$ is a Hilbert space.
\item[(ds2)] $\iota=\{\iota_j\}_{j\in \cJ}$ is a family of linear isometries $\iota_j\colon 
\cH_j\ra\cH$ such that $\iota_l^*\iota_j=0$ for all $j\neq l$.
\item[(ds3)] $\{\iota_j h\mid j\in \cJ,\ h\in\cH_j\}$ is total in $\cH$.
\end{itemize}
\begin{remark}\label{r:dirsum} Orthogonal direct sums of any bundle of Hilbert spaces $\cJ\ast\bH$ exist and 
are unique, modulo a certain unitary operator. 
More precisely, let 
$\cH=\bigoplus\limits_{j\in \cJ}\cH_j$ denote the collection of all sections $(h_j)_{j\in \cJ}\in\cF_\bH(\cJ)$ 
subject to the condition that
\begin{equation*}\sum_{j\in \cJ}\|h_j\|_{\cH_j}^2<\infty,
\end{equation*}
with convergence in the sense of summability.
A standard argument of Hilbert spaces shows that if $(h_j)_{j\in J}\in\cH$ then $h_j=0$ 
for all less an at most countable set of $j\in \cJ$. Then, if $h=(h_j)_{j\in \cJ}$ and 
$k=(k_j)_{j\in \cJ}$ are arbitrary sections in $\cH$ then, letting
\begin{equation}\label{e:laheka}
\langle h,k\rangle_0=\sum_{j\in J}\langle h_j,k_j\rangle_{\cH_j},
\end{equation}
correctly defines an inner product in $\cH$ with respect to which it becomes a Hilbert 
space and the summation is actually an absolutely convergent series.

Let $j\in \cJ$ and $\iota_j\colon \cH_j\ra \cH$ be the natural inclusion map defined for each $h \in \cH_{j}$ as, 
$\iota_j h=\widehat h \in \mathcal{F}_{\cH}(\mathcal{J})$ where, 
\begin{equation*}
\widehat{h}(l): = \left\{\begin{array}{cc}
h, & \mbox{if}\; l = j;\\
0, & \mbox{otherwise.}
\end{array}\right.
\end{equation*}
Then $\iota_j$ is an isometry for each $j \in \mathcal{J}$. Further, it follows from \eqref{e:laheka} that
$\iota_l^*\iota_j=0$, whenever $l\neq j$, and $\{\iota_j h\mid j\in \cJ,\ h\in\cH_j\}$ is 
total in $\cH$.

Suppose $(\cK;\kappa)$ is another direct sum of the bundle $\{\cH_j\}_{j\in \cJ}$ of Hilbert 
spaces and let $U\colon \cH\ra\cK$ be defined as follows: for each $j\in \cJ$ and $h\in\cH_j$,  
let $U\iota_j h=\kappa_j h$ and then extend it by linearity to the span of 
$\{\iota_j h\mid  j\in \cJ,\ h\in\cH_j\}$. It is easy to see that $U$ is isometric, densely 
defined in $\cH$ and the range is dense in $\cK$. Hence $U$ extends uniquely to a unitary 
operator (again denoted by $U$) $U\colon \cH\ra\cK$ such that $U\iota_j=\kappa_j$, for all $j\in \cJ$.
\end{remark}
\begin{remark}\label{r:iovak} Given a bundle of Hilbert spaces $\bH=\{\cH_j\}_{j\in \cJ}$,
we observe that a direct sum of $\bH$ is nothing else than a  ``minimal 
linearisation" of the ``identity operator $\bH$ valued kernel", 
in the sense of Subsection 2.1 in \cite{GheondeaTilki}, that is,
the kernel $I$ defined by $I(j,j)=I_{\cH_j}$, the identity operator on $\cH_j$, 
and $I(j,l)=0$ for all $j\neq l$. 
\end{remark}

As Remark~\ref{r:dirsum} shows, the definition provided by the conditions (d1) through (d3) is not enough for
uniqueness. We need some definitions that will be used more intensively in Subsection~\ref{ss:bbhs}.
A \emph{special orthogonal direct sum} of $\bH$ is, 
by definition, a Hilbert space $\cR$ subject to the following conditions.
\begin{itemize}
\item[(sd1)] $\cR\subseteq \cF_\bH(\cJ)$ with all algebraic operations and it is a Hilbert space.
\item[(sd2)] For all $j\in\cJ$ and all $h\in\cH_j$, the section $\widehat h\in\cR$, with notation as 
in Remark~\ref{r:dirsum}.
\item[(sd3)] For all $f\in\cR$, $j\in\cJ$ and $h\in\cH_j$, we have $\langle f(j),h\rangle_{\cH_j}
=\langle f,\widehat h\rangle_{\cR}$. 
\end{itemize} 

\begin{remark}\label{r:repo}
(a) The definition of a special orthogonal direct sum of a bundle $\bH$ of Hilbert spaces is actually a 
``reproducing kernel Hilbert space", in the sense specified in Subsection~2.1 in \cite{GheondeaTilki}, 
associated to the ``identity operator $\bH$ valued kernel" defined as in 
Remark~\ref{r:iovak}. Note that the condition (sd3) is called the \emph{reproducing property}.

(b) It is easy to see that, if $\cR$ is a special orthogonal direct sum of $\bH$, then the following property holds.
\begin{itemize}
\item[(sd4)] The linear span of $\left\{\widehat h\mid j\in\cJ\mbox{ and }h\in\cH_j\right\}$ is dense in $\cR$.
\end{itemize}

(c) If $\cR$ is a special orthogonal direct sum of $\bH$, for arbitrary $j\in\cJ$, letting $\iota_j(h)=\widehat h$ 
for each $h\in\cH_j$, then the pair $(\cR,\iota)$ is an orthogonal direct sum of $\bH$.
\end{remark}

\begin{proposition}\label{p:resal} Given any bundle $\bH=\{\cH_j\}_{j\in\cJ}$ of Hilbert spaces, there exists a unique special orthogonal direct sum of $\bH$.
\end{proposition}

\begin{proof} We leave to the reader to verify that the Hilbert space $\cH=\bigoplus\limits_{j\in\cJ}\cH_j$ defined in
Remark~\ref{r:dirsum} is a special orthogonal sum of the bundle $\cJ\ast \bH$. For uniqueness, let us observe 
that the vector space $\cF_\bH(\cJ)$ has a Hausdorff locally convex topology given by the seminorms 
$p_j(f)=\|f(j)\|_{\cH_j}$, $f\in\cF_\bH(\cJ)$, $j\in\cJ$. Let $\cR$ be an arbitrary special direct sum of the bundle
$\cJ\ast\bH$ and note that, with respect to these specified topologies, the embedding 
$\cR\hookrightarrow \cF_\bH(\cJ)$ is continuous. Equivalently, this means that convergence in the norm of $\cR$ 
implies pointwise convergence.

In view of the property (sd2), it is easy to observe that the vector space defined by
$\lin\left\{\widehat h\mid h\in \cH_j,\ j\in\cJ\right\}$ is contained in both $\cH$ and $\cR$. Further, it follows from 
Remark~\ref{r:repo}.(b) that it is dense in both of $\cH$ and $\cR$. Also, in view of the reproducing property (sd3), the norms 
$\|\cdot\|_{\cH}$ and $\|\cdot\|_{\cR}$ coincide on  $\lin\left\{\widehat h\mid h\in \cH_j,\ j\in\cJ\right\}$. Since both $\cH$ 
and $\cR$ are complete and their norm topologies are stronger than the pointwise convergence topology of 
$\cF_\bH(\cJ)$, it follows that $\cH=\cR$.
\end{proof}

We need the following proposition that was pointed out in \cite{Gheondea4}. We provide a proof 
because it clarifies some issues although, with the specifications provided before, it is almost clear.

\begin{proposition}\label{p:dirsum} Let $\Lambda$ be a directed set and let $\cJ$ be an arbitrary index set. 
Assume that, for each $j\in\cJ$, there is given a strictly inductive system of Hilbert spaces 
$(\cH_{\lambda,j})_{\lambda\in\Lambda}$. Then $(\bigoplus\limits_{j\in\cJ} \cH_{\lambda,j})_{\lambda\in\Lambda}$,
the net of special orthogonal direct sums,
is a strictly inductive system of Hilbert spaces. In addition, if for any $j\in\cJ$, the strictly inductive system of
Hilbert spaces $(\cH_{\lambda,j})_{\lambda\in\Lambda}$ is representing, then 
$(\bigoplus\limits_{j\in\cJ} \cH_{\lambda,j})_{\lambda\in\Lambda}$
is a representing strictly inductive system of Hilbert spaces as well.
\end{proposition}

\begin{proof} In view of the uniqueness of the special orthogonal direct sum $\bigoplus\limits_{j\in\cJ} \cH_{\lambda,j}$, 
see Proposition~\ref{p:resal}, and the assumption that $(\cH_{\lambda,j})_{\lambda\in\Lambda}$ is a strictly inductive limit 
for each $j\in\cJ$, for 
each $j\in\cJ$, it is easy to see that, whenever $\lambda\leq \epsilon$, we have the isometric
inclusion $\bigoplus\limits_{j\in\cJ} \cH_{\lambda,j}\subseteq \bigoplus\limits_{j\in\cJ} \cH_{\epsilon,j}$.

Assume that, for any $j\in\cJ$, the strictly inductive system of
Hilbert spaces $(\cH_{\lambda,j})_{\lambda\in\Lambda}$ is representing. Let $\lambda,\nu$ be arbitrary in 
$\Lambda$ and let $\epsilon\in\Lambda$ be such that $\lambda,\nu\leq \epsilon$. For each $j$, the orthogonal
projections $P_{\lambda,\epsilon,j},P_{\nu,\epsilon,j}\in\cB(\cH_{\epsilon,j})$, on $\cH_{\lambda,j}$ and 
$\cH_{\nu,j}$, respectively, commute and observe that,
\begin{equation*}
P_{\lambda,\epsilon}:=\bigoplus_{j\in \cJ} P_{\lambda,\epsilon,j},\quad P_{\nu,\epsilon}
:=\bigoplus_{j\in\cJ} P_{\nu,\epsilon,j},
\end{equation*}
are the orthogonal projections of $\bigoplus\limits_{j\in\cJ} \cH_{\epsilon,j}$ onto $\bigoplus\limits_{j\in\cJ} \cH_{\lambda,j}$
and $\bigoplus_{j\in\cJ} \cH_{\nu,j}$, respectively, and they commute.
\end{proof}

With notation and assumptions as in Proposition~\ref{p:dirsum}, if 
$\cH_j=\varinjlim\limits_{\lambda\in\Lambda}\cH_{\lambda,j}$ for each $j\in\cJ$, we define the 
\emph{orthogonal direct sum} of the locally Hilbert spaces $\cH_j$ by
\begin{equation*}
\bigoplus_{j\in\cJ}\cH_j:=\varinjlim_{\lambda\in\Lambda} \bigoplus_{j\in\cJ} \cH_{\lambda,j}.
\end{equation*}
If all locally Hilbert spaces $\cH_j$ are representing, then the locally Hilbert space $\bigoplus\limits_{j\in\cJ}\cH_j$ is
representing.

\subsection{Locally Bounded Operators}\label{ss:lbo}
Let $\cH=\varinjlim\limits_{\lambda\in \Lambda}\cH_\lambda$ and 
$\cK=\varinjlim\limits_{\lambda\in \Lambda}\cK_\lambda$ be two locally Hilbert spaces generated by strictly
inductive systems of Hilbert spaces 
$\left((\cH_\lambda)_{\lambda\in\Lambda},(J_{\nu,\lambda}^\cH)_{\lambda\leq\nu}\right)$ 
and, respectively, 
$\left((\cK_\lambda)_{\lambda\in\Lambda},(J_{\nu,\lambda}^\cK)_{\lambda\leq\nu}\right)$,
indexed on the same directed set $\Lambda$. Let $J^{\mathcal{H}}_{\lambda}\colon \mathcal{H}_{\lambda} \to \mathcal{H}$ and 
$J^{\mathcal{K}}_{\lambda}\colon \mathcal{K}_{\lambda} \to \mathcal{K}$ be inclusion maps for each $\lambda \in \Lambda$.  
A linear map $T\colon \cH\ra\cK$ is called a \emph{locally bounded operator} if it has the following properties.
\begin{itemize}
\item[(lbo1)] There exists a net of operators $(T_\lambda)_{\lambda\in\Lambda}$, 
with $T_\lambda\in\cB(\cH_\lambda,\cK_\lambda)$ and such that 
$TJ^\cH_\lambda=J_\lambda^\cK T_\lambda$, for all $\lambda\in\Lambda$. 
\item[(lbo2)] The net of operators $(T_\lambda^*)_{\lambda\in\Lambda}$ 
is coherent as well, that is, $T_\nu^* J_{\nu,\lambda}^\cK=J_{\nu,\lambda}^\cH T_\lambda^*$,
for all $\lambda,\nu\in \Lambda$ such that $\lambda\leq\nu$.
\end{itemize}
We denote by $\Lloc(\cH,\cK)$ the collection of all locally bounded operators 
$T\colon\cH\ra\cK$. We observe that $\Lloc(\cH,\cK)$ is a vector space and that
there is a canonical embedding
\begin{equation}\label{e:lbl} 
\Lloc(\cH,\cK)\subseteq\varprojlim_{\lambda\in\Lambda}\cB(\cH_\lambda,\cK_\lambda).
\end{equation}

The correspondence between $T\in\Lloc(\cH,\cK)$ and the net of operators 
$(T_\lambda)_{\lambda\in\Lambda}$ as in (lbo1) and (lbo2) is unique. Given $T\in\Lloc(\cH,\cK)$,
for arbitrary $\lambda\in\Lambda$ we have $T_\lambda h=Th$, for all $h\in\cH_\lambda$, with the
observation that $Th\in\cK_\lambda$. Conversely, if $(T_\lambda)_{\lambda\in\Lambda}$ is a 
net of operators $T_\lambda\in\cB(\cH_\lambda,\cK_\lambda)$ satisfying (lbo2) then, letting
$Th=T_\lambda h$ for arbitrary $h\in\cH$, where $\lambda\in\Lambda$ is such that 
$h\in\cH_\lambda$, it follows that $T$ is a locally bounded operator: 
this definition is correct by (lbo2). We will use the notation 
\begin{equation}\label{e:tpl} 
T=\varprojlim\limits_{\lambda\in\Lambda}T_\lambda.\end{equation}

The following result tells us which additional properties a net of bounded operators 
must have in order to produce a locally bounded operator, see Proposition~1.12 in \cite{Gheondea3}.

\begin{proposition}\label{p:lbo2} Let
$(T_\lambda)_{\lambda\in\Lambda}$ be a net with $T_\lambda\in\cB(\cH_\lambda,\cK_\lambda)$ ,
for all $\lambda\in\Lambda$. The following assertions are equivalent.

\nr{1} For every $\lambda,\nu\in\Lambda$ such that $\lambda\leq\nu$ we have
\begin{equation}\label{e:temuha}
T_\nu|_{\cH_\lambda}=J_{\nu,\lambda}^\cK T_\lambda,\mbox{ and }
T_\nu P_{\lambda,\nu}^\cH=P_{\lambda,\nu}^\cK T_\nu,\end{equation} 
where $P^\cH_{\lambda,\nu}$ is the orthogonal projection of $\cH_\nu$ onto its subspace 
$\cH_\lambda$.

\nr{2} For every $\lambda,\nu\in\Lambda$ such that $\lambda\leq\nu$, with respect to the
decompositions
\begin{equation*}\cH_\nu=\cH_\lambda\oplus(\cH_\nu\ominus\cH_\lambda),\quad
\cK_\nu=\cK_\lambda\oplus(\cK_\nu\ominus\cK_\lambda),
\end{equation*} the operator $T_\nu$ has the following block matrix representation
\begin{equation}\label{e:temule} 
T_\nu=\left[\begin{matrix} T_\lambda & 0 \\ 0 & T_{\lambda,\nu}\end{matrix}\right],
\end{equation}
for some bounded linear operator $T_{\lambda,\nu}\colon\cH_\nu\ominus\cH_\lambda\ra
\cK_\nu\ominus\cK_\lambda$.

\nr{3} There exists an operator $T\in\Lloc(\cH,\cK)$ such that 
$T|_{\cH_\lambda}=J^\cK_\lambda T_\lambda$ for all $\lambda\in\Lambda$.

In addition, if any of these assertions holds (hence all of them hold), 
the operator $T\in\Lloc(\cH,\cK)$ as in \emph{(3)} is uniquely determined by 
$(T_\lambda)_{\lambda\in\Lambda}$ and
\begin{equation*} T=\varprojlim_{\lambda\in\Lambda}T_\lambda.
\end{equation*}
\end{proposition}

As a consequence of the previous proposition, one can introduce the adjoint operation on 
$\Lloc(\cH,\cK)$. Let $T=\varprojlim\limits_{\lambda\in\Lambda}T_\lambda\in\Lloc(\cH,\cK)$ and
hence the net $(T_\lambda)_{\lambda\in\Lambda}$ satisfies the conditions \eqref{e:temuha}. Then,
consider the net of bounded operators $(T_\lambda^*)_{\lambda\in\Lambda}$, 
$T_\lambda^*\in\cB(\cH_\lambda,\cK_\lambda)$ for all $\lambda\in\Lambda$ and observe that,
for all $\lambda,\nu\in\Lambda$ with $\nu\geq \lambda$, we have
\begin{equation}
T_\nu^*|_{\cK_\lambda}=J_{\nu,\lambda}^\cH T_\lambda^*\mbox{ and }
T_\nu^* P_{\lambda,\nu}^\cK=P_{\lambda,\nu}^\cH T_\nu^*,\end{equation}
hence, by Proposition~\ref{p:lbo2}, there exists a unique operator $T^*\in\Lloc(\cK,\cH)$ such that
\begin{equation}\label{e:temuhadj}
T^*=\varprojlim_{\lambda\in\Lambda}T_\lambda^*.
\end{equation}

Given three locally Hilbert spaces $\cH=\varinjlim\limits_{\lambda\in\Lambda}\cH_\lambda$,
$\cK=\varinjlim\limits_{\lambda\in\Lambda}\cK_\lambda$, and 
$\cG=\varinjlim\limits_{\lambda\in\Lambda}\cG_\lambda$, indexed on the same directed set $\Lambda$, we
observe that the composition of locally bounded operators yields locally bounded operators,
more precisely, whenever $T\in\Lloc(\cH,\cK)$ and $S\in\Lloc(\cK,\cG)$ it follows that 
$ST\in\Lloc(\cH,\cG)$ and the usual algebraic properties as associativity and distributivity with
respect to addition and multiplication hold. Moreover, for each $\lambda\in\Lambda$, 
we have
\begin{equation*} (ST)_\lambda =S_\lambda T_\lambda
\end{equation*}
and hence
\begin{equation}\label{e:setep} ST=\varprojlim_{\lambda\in\Lambda} S_\lambda T_\lambda.
\end{equation}
In addition, composition of locally bounded
operators behaves as usually with respect to the adjoint operation, that is,
\begin{equation}\label{e:dadj}
(ST)^*=T^*S^*,\quad T\in\Lloc(\cH,\cK),\ S\in\Lloc(\cK,\cG).
\end{equation} 

A locally bounded operator 
$T=\varprojlim\limits_{\lambda\in\Lambda}T_\lambda\in\Lloc(\cH,\cK)$ and the concept of composition of locally 
bounded operators, it is natural to call it \emph{locally isometric} if 
$T^*T=I_{\cH}$. It is easy to see that $T$ is locally isometric if and only if 
$T_\lambda\in\cB(\cH_\lambda,\cK_\lambda)$ is isometric for all $\lambda\in \Lambda$. 
Also, $T$ is called \emph{locally unitary} if both $T$ and $T^*$ are locally isometric and then, it is clear that $T$ 
is locally unitary if and only if $T_\lambda\in\cB(\cH_\lambda,\cK_\lambda)$ is unitary for all 
$\lambda\in\Lambda$.

Also, any locally bounded operator $T\colon\cH\ra\cK$ 
is continuous with respect to the inductive limit topologies of $\cH$ and $\cK$. However, 
in general, a locally bounded operator $T\colon\cH\ra\cK$
may not be continuous with respect to the norm topologies of $\cH$ and $\cK$. An arbitrary
linear operator $T\in\Lloc(\cH,\cK)$ is continuous with respect to the norm topologies of $\cH$ 
and $\cK$ if and only if, with respect to the notation as in (lbo1) and (lbo2), 
$\sup_{\lambda\in\Lambda}\|T_\lambda\|_{\cB(\cH_\lambda,\cK_\lambda)}<\infty$. In this case,
the operator $T$ uniquely extends to an operator $\widetilde T\in\cB(\widetilde\cH,\widetilde\cK)$, 
where
$\widetilde\cH$ and $\widetilde\cK$ are the Hilbert space completions of $\cH$ and $\cK$, respectively, 
and $\|\widetilde T\|
=\sup\limits_{\lambda\in\Lambda}\|T_\lambda\|_{\cB(\cH_\lambda,\cK_\lambda)}$.

As a consequence of \eqref{e:lbl}, $\Lloc(\cH,\cK)$ has a natural locally convex topology
induced by the projective limit locally convex topology of 
$\varprojlim\limits_{\lambda\in\Lambda}\cB(\cH_\lambda,\cK_\lambda)$, more precisely, 
generated by the family of seminorms $\{q_\lambda\}_{\lambda\in\Lambda}$ defined by
\begin{equation}\label{e:qmt}
q_\nu(T)=\|T_\nu\|_{\cB(\cH_\nu,\cK_\nu)},\quad T
=(T_\lambda)_{\lambda\in\Lambda}\in\varprojlim_{\lambda\in\Lambda}
\cB(\cH_\lambda,\cK_\lambda),\quad \lambda\in\Lambda.
\end{equation}
With respect to the embedding \eqref{e:lbl}, $\Lloc(\cH,\cK)$ is closed
in $\varprojlim\limits_{\lambda\in\Lambda}\cB(\cH_\lambda,\cK_\lambda)$, hence complete.
               
The locally convex space $\Lloc(\cH,\cK)$ can be organised as a projective limit of locally 
convex spaces, in view of \eqref{e:lbl}, more precisely, letting 
$\pi_\nu\colon \varprojlim\limits_{\lambda\in\Lambda} \cB(\cH_\lambda,\cK_\lambda)\ra 
\cB(\cH_\nu,\cK_\nu)$, for $\nu\in\Lambda$, 
be the canonical projection, then
\begin{equation}\label{e:pll}
\Lloc(\cH,\cK)=\varprojlim_{\lambda\in\Lambda} \pi_\lambda (\Lloc(\cH,\cK)).
\end{equation}

\subsection{Locally $C^*$-Algebras}\label{ss:lca}
A $*$-algebra $\cA$ is called a \emph{locally $C^*$-algebra}
if it has a complete Hausdorff locally convex topology which is induced by a family of 
$C^*$-seminorms, that is,
seminorms $p$ with the property $p(a^*a)=p(a)^2$ for all $a\in\cA$, see \cite{Inoue}. 
Any $C^*$-seminorm $p$ has also the properties $p(a^*)=p(a)$ and $p(ab)\leq p(a)p(b)$ 
for all $a,b\in\cA$, cf.\ \cite{Sebestyen}. 
Locally $C^*$-algebras have been also called \emph{$LMC^*$-algebras} 
\cite{Schmudgen}, \emph{$b^*$-algebras} \cite{Allan}, \emph{pro $C^*$-algebras} 
\cite{Voiculescu}, \cite{Phillips}, and \emph{multinormed $C^*$-algebras} \cite{Dosi1}. 

If $\cA$ is a locally $C^*$-algebra, let $S(\cA)$ denote the collection of all continuous 
$C^*$-seminorms and note that $S(\cA)$ is a directed set with respect to the partial order defined as
follows:
$p\leq q$ if $p(a)\leq q(a)$ for all $a\in\cA$. If $p\in S(\cA)$, then 
\begin{equation}\label{e:cisp} 
\cI_p=\{a\in\cA\mid p(a)=0\}
\end{equation} is a
closed two sided $*$-ideal of $\cA$ and $\cA_p=\cA/\cI_p$ becomes a $*$-algebra with respect 
to the $C^*$-norm $\|\cdot\|_p$ induced by $p$, more precisely,
\begin{equation}\label{e:tlp} 
\|a+\cI_p\|_p=p(a),\quad a\in\cA.
\end{equation} 
It is proven by C.~Apostol in \cite{Apostol} that the $C^*$-norm  is complete and hence $\cA_p$ is a $C^*$-
algebra already and no completion is needed, as assumed in \cite{Schmudgen} and \cite{Phillips}.
Letting $\pi_p\colon \cA\ra\cA_p$ denote the 
canonical projection, for any $p,q\in S(\cA)$ such that $p\leq q$, there exists a canonical 
$*$-epimorphism of $C^*$-algebras $\pi_{p,q}\colon \cA_q\ra \cA_p$ such that 
$\pi_p=\pi_{p,q}\circ \pi_q$, with respect to which 
$(\cA_p)_{p\in S(\cA)}$ becomes a projective system of $C^*$-algebras such that 
\begin{equation}\label{e:alim}
\cA=\varprojlim\limits_{p\in S(\cA)} \cA_p,\end{equation}
see \cite{Schmudgen}, \cite{Phillips}. This projective limit is taken
in the category of locally convex $*$-algebras and hence all the morphisms are continuous
$*$-morphisms of locally convex $*$-algebras.

Letting $b(\cA)=\bigl\{a\in\cA\mid \sup\limits_{p\in S(\cA)} p(a)<+\infty\bigr\}$, it follows that 
$\|a\|=\sup\limits_{p\in S(\cA)} p(a)$ is a $C^*$-norm on the $*$-algebra $b(\cA)$ and, with respect to
this norm, $b(\cA)$ is a $C^*$-algebra, dense in $\cA$, see \cite{Apostol}. The elements
of $b(\cA)$ are called \emph{bounded}.

The $C^*$-algebra $b(\cA)$, together with the system of $*$-morphisms $\{\pi_p\}_{p\in S(\cA)}$, where, with an 
abuse of notation, we denote by the same symbol $\pi_p\colon b(\cA)\ra \cA_p$ 
the canonical projection  $\pi_p\colon \cA\ra\cA_p=\cA/\cI_p$ restricted to $b(\cA)$,
is the projective limit of the projective system of 
$C^*$-algebras $(\cA_p,\pi_p)_{p\in S(\cA)}$ in the category $\fC^*$ of all $C^*$-algebras. 

\subsection{The Locally $C^*$-Algebra $\Lloc(\cH)$.} \label{ss:lcsab}
Let $\cH=\varinjlim\limits_{\lambda\in\Lambda}\cH_\lambda$
be a locally Hilbert space and $\Lloc(\cH)$ be the locally convex space of all locally 
bounded operators $T\colon \cH\ra\cH$,
as in Subsection~\ref{ss:lbo}. Note that
$\Lloc(\cH)$ has a natural product and a natural involution $*$, 
with respect to which it is a $*$-algebra, see \eqref{e:temuhadj} and \eqref{e:setep}.
For each $\mu\in\Lambda$, consider the $C^*$-algebra 
$\cB(\cH_\mu)$ of all bounded linear operators in $\cH_\mu$ and 
$\pi_\mu\colon\Lloc(\cH)\ra\cB(\cH_\mu)$ be the canonical map
\begin{equation}\label{e:pimu}
\pi_\mu(T)=T_\mu,\quad T=\varprojlim\limits_{\lambda\in\Lambda}T_\lambda\in\Lloc(\cH).
\end{equation} Let $\Lloc(\cH_\mu)$ denote the range of $\pi_\mu$ and note that it is 
a $C^*$-subalgebra of $\cB(\cH_\mu)$.
It follows that $\pi_\mu\colon\Lloc(\cH)\ra\Lloc(\cH_\mu)$ is a $*$-morphism of $*$-algebras 
and, for each 
$\lambda,\mu\in \Lambda$ with $\lambda\leq\mu$, 
there is a unique $*$-epimorphism of $C^*$-algebras 
$\pi_{\lambda,\mu}\colon \Lloc(\cH_\mu)\ra\Lloc(\cH_\lambda)$, such that 
$\pi_\lambda=\pi_{\lambda,\mu}\pi_\mu$. More precisely,
$\pi_{\lambda,\mu}$ is the compression of $\cH_\mu$ to $\cH_\lambda$,
\begin{equation}\label{e:plmh} 
\pi_{\lambda,\mu}(S)=J_{\mu,\lambda}^*SJ_{\mu,\lambda},\quad S\in\Lloc(\cH_\mu).
\end{equation}
Then $\left((\Lloc(\cH_\lambda))_{\lambda\in\Lambda},
(\pi_{\lambda,\mu})_{\lambda,\mu\in\Lambda,\, \lambda\leq \mu}\right)$ is
a projective system of $C^*$-algebras, in the sense that,
\begin{equation}\label{e:ple} 
\pi_{\lambda,\eta}=\pi_{\lambda,\mu}\circ\pi_{\mu,\eta},\quad \lambda,\mu,\eta\in\Lambda,\
\lambda\leq\mu\leq\eta,
\end{equation} and, in addition,
\begin{equation}\label{e:coh}
\pi_\mu(S)P_{\lambda,\mu}=P_{\lambda,\mu}\pi_\mu(S),\quad \lambda,\mu\in\Lambda,\ 
\lambda\leq\mu,\ S\in\Lloc(\cH_\mu),
\end{equation}
such that
\begin{equation}\label{e:loclhs} \Lloc(\cH)=\varprojlim_{\lambda\in\Lambda} \Lloc(\cH_\lambda),
\end{equation}
where, the projective limit is considered in the category of locally convex $*$-algebras. 
In particular, $\Lloc(\cH)$ is a locally $C^*$-algebra.

For each $\lambda\in \Lambda$, letting $p_\lambda\colon\Lloc(\cH)\ra\RR$ be defined by 
\begin{equation}\label{e:psmu}
 p_\lambda(T)=\|T_\lambda\|_{\cB(\cH_\lambda)},\quad T=\varprojlim_{\nu\in\Lambda}
T_\nu\in\Lloc(\cH),
\end{equation} then $p_\lambda$ is a $C^*$-seminorm on $\Lloc(\cH)$. Then $\Lloc(\cH)$ 
becomes a unital locally $C^*$-algebra with respect to  the locally convex topology induced by  
$\{p_\lambda\}_{\lambda\in\Lambda}$. 

The $C^*$-algebra $b(\Lloc(\cH))$
coincides with the set of all locally bounded operators 
$T=\varprojlim\limits_{\lambda\in\Lambda}T_\lambda$ such that $(T_\lambda)_{\lambda\in\Lambda}$
is uniformly bounded, i.e.\ $\sup\limits_{\lambda\in\Lambda}\|T_\lambda\|<\infty$, 
equivalently, those locally bounded operators $T\colon \cH\ra\cH$ which are bounded with respect
to the canonical norm $\|\cdot\|_\cH$ on the pre-Hilbert space 
$(\cH,\langle\cdot,\cdot\rangle_\cH)$. In particular, $b(\Lloc(\cH))$ is a $C^*$-subalgebra of 
$\cB(\widetilde\cH)$, where $\widetilde\cH$ denotes the completion of 
$(\cH,\langle\cdot,\cdot\rangle_\cH)$ to a Hilbert space.

By taking the closure, any locally bounded operator $T\in\Lloc(\cH)$ uniquely extends
to a closed and densely defined operator $\widetilde T$ on the Hilbert space $\widetilde\cH$.

\subsection{Weak Topologies on $\Lloc(\cH)$.}\label{ss:wt} 
As a locally convex space, $\Lloc(\cH)$ has its projective limit topology given by the family of seminorms
$\{p_\lambda\}_{\lambda\in\Lambda}$ defined at \eqref{e:psmu}. In this article, we use two other operator
topologies. Briefly, following \cite{Gheondea2} and \cite{Gheondea3}, the \emph{weak operator topology} on $\Lloc(\cH)$ is the 
locally convex topology associated to the family of seminorms
\begin{equation}\label{e:wot} \Lloc(\cH)\ni T\mapsto \langle Th,k\rangle_\cH,\quad h,k\in\cH.
\end{equation}
On the other hand, for each $\lambda\in\Lambda$ there is the weak operator topology 
$\tau_{\lambda,\mathrm{wo}}$ on $\Lloc(\cH_\lambda)$ and 
$\{(\Lloc(\cH_\lambda),\tau_{\lambda,\mathrm{wo}})\}_{\lambda\in\Lambda}$ is a projective system of locally 
convex spaces.

\begin{proposition}\label{p:wot} The weak operator topology on $\Lloc(\cH)$ coincides with the projective
limit topology of the projective system of locally convex spaces
$\{(\Lloc(\cH_\lambda),\tau_{\lambda,\mathrm{wo}})\}_{\lambda\in\Lambda}$. In particular $\Lloc(\cH)$ is complete
with respect to the weak operator topology.
\end{proposition}

\begin{proof} The projective limit topology of the projective system
$\{(\Lloc(\cH_\lambda),\tau_{\lambda,\mathrm{wo}})\}_{\lambda\in\Lambda}$ is, by definition, the weakest
locally convex topology on $\Lloc(\cH)$ that makes all the canonical projections $\pi_\lambda\colon
\Lloc(\cH)\ra (\Lloc(\cH_\lambda),\tau_{\lambda,\mathrm{wo}})$ defined at \eqref{e:pimu} continuous. Since,
for any $h,k\in\cH$ there exists $\lambda\in\Lambda$ such that $h,k\in \cH_\lambda$, it follows that
all seminorms defined at \eqref{e:wot} are continuous with respect to the projective limit topology. The 
converse is clear.
\end{proof}

The \emph{strong operator topology} on $\Lloc(\cH)$ is the locally convex topology
associated to the family of seminorms
\begin{equation}\label{e:sot} \Lloc(\cH)\ni T\mapsto \|Th\|_\cH,\quad h\in\cH.
\end{equation}
On the other hand, for each $\lambda\in\Lambda$ there is the strong operator topology
$\tau_{\lambda,\mathrm{so}}$ on $\Lloc(\cH_\lambda)$ and
$\{(\Lloc(\cH_\lambda),\tau_{\lambda,\mathrm{so}})\}_{\lambda\in\Lambda}$ is a projective system of locally 
convex spaces.

\begin{proposition}\label{p:sot}
The strong operator topology on $\Lloc(\cH)$ coincides with the projective limit topology of the 
projective system of locally convex spaces $\{(\Lloc(\cH_\lambda),\tau_{\lambda,\mathrm{so}})\}_{\lambda\in\Lambda}$. In
particular, $\Lloc(\cH)$ is complete with respect to the strong operator topology.
\end{proposition}

\begin{proof} The projective limit topology of the projective system
$\{(\Lloc(\cH_\lambda),\tau_{\lambda,\mathrm{so}})\}_{\lambda\in\Lambda}$ is, by definition, the weakest
locally convex topology on $\Lloc(\cH)$ that makes all the canonical projections $\pi_\lambda\colon
\Lloc(\cH)\ra (\Lloc(\cH_\lambda),\tau_{\lambda,\mathrm{so}})$ defined at \eqref{e:pimu} continuous. Since,
for any $h\in\cH$ there exists $\lambda\in\Lambda$ such that $h\in \cH_\lambda$, it follows that
all seminorms defined at \eqref{e:sot} are continuous with respect to the projective limit topology. The 
converse is clear.
\end{proof}

\subsection{Locally von Neumann Algebras} \label{ss;lvna}
In this subsection, we briefly recall the notion of a locally von Neumann algebra and review a few of its properties, following
mainly \cite{Dosi3}. We fix a locally Hilbert space $\cH=\varinjlim\limits_{\lambda\in\Lambda}\cH_\lambda$ and we assume it
to be representing (commutative). Although the condition that the locally Hilbert space is representing is not necessary for some 
definitions, the main results require it.

\begin{definition}\label{d:lvn} $\cM$ is
a \emph{locally von Neumann algebra} on $\cH$ if the following properties hold.
\begin{itemize}
\item[(LVN1)] $\cM$ is a $*$-subalgebra of $\Lloc(\cH)$.
\item[(LVN2)] $\cM$ is closed with respect to the 
strong operator topology, equivalently, with respect to the weak operator topology.
\item[(LVN3)] For each $\lambda\in \Lambda$ we have $\cM P_\lambda\in \cM$, where $P_\lambda$ denotes the unique 
Hermitian projection of $\cH$ onto $\cH_\lambda$ as in Definition \ref{d:rep}.
\item[(LVN3)] $\cM$ contains the identity.
\end{itemize}
\end{definition}

\begin{remark}\label{r:lvna}
(a) Let $\cH=\varinjlim\limits_{\lambda\in\Lambda}\cH_\lambda$ be a locally Hilbert space and let $\Lloc(\cH)$ be the unital 
locally $C^*$-algebra as introduced in Subsection~\ref{ss:lcsab}. With this definition, we first observe that, in view of 
Proposition~\ref{p:wot} and Proposition~\ref{p:sot}, $\Lloc(\cH)$ is a locally von Neumann algebra. 

(b) Any locally von Neumann algebra on $\cH$ is a unital locally $C^*$-subalgebra of $\Lloc(\cH)$.

(c) Also, with notation as in
Subsection~\ref{ss:lcsab}, let us observe that, for each $\mu\in\Lambda$, $\cB_{\mathrm{loc}}(\cH_\mu)$, the range of $\pi_\mu$ 
defined as in \eqref{e:pimu} is a von Neumann algebra on the Hilbert space $\cH_\mu$. In addition, for each 
$\lambda,\mu\in\Lambda$, the map $\pi_{\lambda,\mu}\colon\cB_{\mathrm{loc}}(\cH_\mu)\ra \cB_{\mathrm{loc}}(\cH_\lambda)$
defined as in \eqref{e:plmh}, is weakly continuous $*$-epimorphism of von Neumann algebras and then 
$\left((\Lloc(\cH_\lambda))_{\lambda\in\Lambda},
(\pi_{\lambda,\mu})_{\lambda,\mu\in\Lambda,\, \lambda\leq \mu}\right)$ is
a projective system of von Neumann algebras, in the sense that \eqref{e:ple} holds and, in addition, \eqref{e:coh} holds as well. 
Then, by \eqref{e:loclhs}, it follows that $\Lloc(\cH)$ is a projective limit, considered in the category of locally convex 
$*$-algebras, of von Neumann algebras.

(d) The projective system of von Neumann algebras $\left((\Lloc(\cH_\lambda))_{\lambda\in\Lambda},
(\pi_{\lambda,\mu})_{\lambda,\mu\in\Lambda,\, \lambda\leq \mu}\right)$ can be taken in the category of von Neumann algebras 
and, in this case, it is $b(\Lloc(\cH))$.
\end{remark}

A locally von Neumann algebra is called a \emph{multinormed $W^\ast$-algebra} in \cite{Fragoulopoulou}, 
where it is defined as a projective  limit of von Neumann algebras. To be more precise, with notation from Subsection~\ref{ss:lca},
that means that, given a locally Hilbert space $\cH=\varinjlim_{\lambda\in\Lambda}\cH_\lambda$, 
for each $\lambda\in \Lambda$,
there exists a von Neumann algebra $\cM_\lambda\in\cB(\cH_\lambda)$ such that, for each $\lambda\leq\mu$, there exists 
$\pi_{\lambda,\mu}\colon\cM_\mu\ra\cM_\lambda$ a $*$-epimorphism of von Neumann algebras, such that 
$\big((\cM_\lambda)_{\lambda\in\Lambda},(\pi_{\lambda,\mu})_{\lambda\leq \mu}\big)$ is a projective system of von Neumann
algebras and $\cM=\varprojlim_{\lambda\in\Lambda} \cM_\lambda$. Further, $\cM$ is an Abelian locally von Neumann algebra if 
and only if it can be realised as a projective limit of Abelian von Neumann algebras.

In this paper, we will use the following equivalent characterisation of locally von Neumann algebras, 
cf.\ Subsection 5.3 in \cite{Dosi3}. 

\begin{theorem} \label{t:lvna}
Let $\cM$ be a locally $C^\ast$-subalgebra of $\Lloc(\cH)$, where $\cH$ is a representing locally Hilbert space. Then, $\cM$ is a 
locally von Neumann algebra if and only if $\cM$ is a multinormed $W^\ast$-algebra. 
\end{theorem}

 Let $\cA$ be any subset of $\Lloc(\cH)$. Then the commutant of $\cA$ is denoted by $\cA^\prime$ and it is defined, 
 as usually, as 
\begin{equation*}
\cA^\prime := \left\{  T \in \Lloc(\cH) \; \; | \; \; TS = ST, \; \; \; \text{for all} \; \; S \in \cA \right\}.
\end{equation*}
The following double commutant theorem is proven in \cite{Dosi3}.

\begin{theorem} \label{t:dc}
Let $\cM$ be a $\ast$-subalgebra of $\Lloc(\cH)$, such that  it contains the identity $I_\cH$ and the locally Hilbert space is 
representing. Then $\cM$ is a locally von Neumann algebra if and 
only if $(\cM^{\prime})^\prime = \cM$.
\end{theorem}

\subsection{Strictly Inductive Systems of Measure Spaces}\label{ss:sisms}
Given a directed set $(\Lambda;\leq)$, a net $((X_\lambda,\Omega_\lambda))_{\lambda
\in\Lambda}$ is called a \emph{strictly inductive system of measurable spaces} if the following properties hold.
\begin{itemize}
\item[(sim1)] For each $\lambda\in\Lambda$, $(X_\lambda,\Omega_\lambda)$ is a measurable space.
\item[(sim2)] For each $\lambda,\nu\in\Lambda$ with $\lambda\leq\nu$ we have 
$X_\lambda\subseteq X_\nu$ and 
$\Omega_\lambda=\{A\cap X_\lambda\mid A\in\Omega_\nu\}\subseteq\Omega_\nu$.
\end{itemize}
Given a strictly inductive system of measurable spaces 
$((X_\lambda,\Omega_\lambda))_{\lambda
\in\Lambda}$, we denote
\begin{equation}\label{e:lms}
X=\bigcup_{\lambda\in\Lambda} X_\lambda,\quad 
\Omega=\bigcup_{\lambda\in\Lambda}\Omega_\lambda.
\end{equation}
The pair $(X,\Omega)$ is called the \emph{inductive limit} of the strictly inductive system of measurable 
spaces, and we use the notation
\begin{equation}\label{e:ilm}
(X,\Omega)=\varinjlim_{\lambda\in\Lambda}(X_\lambda,\Omega_\lambda).
\end{equation}

Note that $\Omega$ is a \emph{ Boolean ring} of subsets in $X$, that is, 
for any $A,B\in\Omega$ it follows 
$A\setminus B, A\cap B, A\cup B\in\Omega$, but, in general, not even a $\sigma$-ring. However, $\Omega$ is a 
\emph{locally $\sigma$-ring} in the sense that, if $(A_n)_{n\in \mathbb{N}}$ is a sequence of subsets from $\Omega$ 
such that there exists $\lambda\in\Lambda$ with the property that $A_n\in\Omega_\lambda$ for all
$n\in\NN$, it follows that $\bigcup\limits_{n\in\NN} A_n$ and $\bigcap\limits_{n\in\NN} A_n$ are in $\Omega$.

We see that $\Omega$ has a canonical extension to a $\sigma$-algebra. Letting 
\begin{equation}\label{e:tomega} \widetilde\Omega:=\{A\subseteq X\mid A\cap 
X_\lambda\in\Omega_\lambda\mbox{ for all }\lambda\in\Lambda\},
\end{equation}
then, by Proposition~3.1 in \cite{Gheondea3},
$\widetilde\Omega$ is a $\sigma$-algebra and $\Omega\subseteq\widetilde\Omega$. It is a fact, see 
Remark~3.4 in \cite{Gheondea4}, that the pair $(X,\tilde\Omega)$ is the inductive limit of the strictly inductive 
system of measurable spaces $((X_\lambda,\Omega_\lambda))_{\lambda
\in\Lambda}$ in the category of measurable spaces.

Given a directed set $(\Lambda;\leq)$, a net $((X_\lambda,\Omega_\lambda,\mu_\lambda))_{\lambda
\in\Lambda}$ of measure spaces is called a \emph{strictly inductive system of measure spaces} if
$(X_\lambda,\Omega_\lambda)_{\lambda\in\Lambda}$ is a strictly inductive system of measurable spaces 
and, in addition,
\begin{itemize}
\item[(sim3)] for any $\lambda,\nu\in\Lambda$ with $\lambda\leq\nu$ and any $U\in\Omega_\lambda$ we 
have $\mu_\lambda(U)=\mu_\nu(U)$.
\end{itemize}
With notation as in \eqref{e:lms},
let $\mu\colon\Omega\ra [0,+\infty]$ be defined by
\begin{equation}\label{e:lmsa}
\mu(\Delta)=\mu_\lambda (\Delta),\quad \Delta\in\Omega,
\end{equation}
where $\lambda\in\Lambda$ is such that $\Delta\in\Omega_\lambda$. 

It is easy to see that $\mu$ is correctly defined, that is, its definition does not depend on $\lambda$. 
In addition, $\mu$ is nondecreasing, 
$\mu(\emptyset)=0$, and \emph{additive} in the sense that whenever $A,B\in\Omega$ are disjoint then 
$\mu(A\cup B)=\mu(A)+\mu(B)$. In general, $\mu$ is not $\sigma$-additive, but it
is \emph{locally $\sigma$-additive}, that is: 
for any mutually disjoint sequence $(A_n)_{n\in\NN}$ such that there exists 
$\lambda\in\Lambda$ with the property that $A_n\in\Omega_\lambda$ for all $n\in\NN$ we 
have
\begin{equation}\label{e:sap}
\mu(\bigcup_{n\in\NN} A_n)=\sum_{n\in\NN}\mu(A_n).
\end{equation}

\begin{definition}
The triple $(X,\Omega,\mu)$ is called the \emph{inductive limit} of the strictly inductive system of measure
spaces $(X_\lambda,\Omega_\lambda,\mu_\lambda)_{\lambda\in\Lambda}$ and we use the notation
\begin{equation}\label{e:ilsm}
(X,\Omega,\mu)=\varinjlim_{\lambda\in\Lambda} (X_\lambda,\Omega_\lambda,\mu_\lambda).
\end{equation}
By $L^\infty_{\mathrm{loc}}(X,\mu)$ we denote
the set of all functions $\phi\colon X\ra \CC$, identified $\mu$-a.e., such that, 
for any $\lambda\in\Lambda$ we have $\phi|_{X_\lambda}\in L^\infty(X_\lambda,\mu_\lambda)$.
\end{definition} 

\begin{remark}\label{r:lef}
In Example~3.5 in \cite{Gheondea4} it is shown that, 
with natural addition, multiplication, multiplication with scalars, and involution,
$L^\infty_{\mathrm{loc}}(X,\mu)$ is a commutative $*$-algebra and,  letting, 
for each $\lambda\in \Lambda$ the seminorm
\begin{equation}\label{e:loclisem}
p_\lambda(\phi):= \|\phi|X_\lambda\|_\infty=\esssup_{x\in X_\lambda}|\phi(x)|,\quad 
\phi\in L^\infty_{\mathrm{loc}}(X,\mu),
\end{equation}
with respect to the topology defined by the family of seminorms 
$\{p_\lambda\}_{\lambda\in\Lambda}$, 
$L^\infty_{\mathrm{loc}}(X,\mu)$ is a locally $C^*$-algebra. More precisely, for each $\lambda\in \Lambda$, let
\begin{equation}\label{e:lefal}
L^\infty_\lambda(X,\mu):=\{\phi \in L^\infty_{\mathrm{loc}}(X,\mu) \mid \supp(\phi)\subseteq X_\lambda\},
\end{equation}
and observe that we have 
a projective system of $C^*$-algebras made up by the $C^*$-algebras
$(L^\infty_\lambda(X,\mu))_{\lambda\in\Lambda}$ and the projective system of $*$-morphisms
$(\Phi_{\lambda,\nu})_{\lambda\leq \nu}$ where, for $\lambda\leq \nu$, the $*$-morphism 
$\Phi_{\lambda,\nu}\colon L^\infty_{\nu}(X,\mu)\ra L^\infty_{\lambda}(X,\mu)$ is defined by $\phi\mapsto 
\phi_{\lambda}$, where $\phi_\lambda(x)=\phi(x)$, if $x\in X_\lambda$, and $\phi_\lambda(x)=0$, if 
$x\in X\setminus X_\lambda$. Then, the projective limit of this projective system of $C^*$-algebras, in the category of locally 
convex $*$-algebras, is $L^\infty_{\mathrm{loc}}(X,\mu)$.
\end{remark}

Although an inductive limit of a strictly inductive system of measure spaces is not, in general, a measure 
space, it has a canonical extension to a measure space, see \cite{KulkarniPamula1} and \cite{Gheondea4}.
Consider the $\sigma$-algebra $\widetilde\Omega$, defined as in \eqref{e:lms}, and
define $\widetilde\mu\colon\widetilde\Omega\ra [0,+\infty]$ by
\begin{equation}\label{e:wimu}
\widetilde\mu(A):=\sup_{\lambda\in\Lambda} \mu_\lambda(A\cap X_\lambda),\quad 
A\in\widetilde\Omega.
\end{equation}
More precisely, if $A\in\widetilde\Omega$ is arbitrary, then, in view of the properties (sim1)--(sim3), the 
net $(\mu_\lambda(A\cap X_\lambda))_{\lambda\in\Lambda}$ is nondecreasing, hence
\begin{equation*}
\widetilde\mu(A)=\begin{cases}+\infty, & \mbox{ if }(\mu_\lambda(A\cap X_\lambda))_{\lambda\in\Lambda}\mbox{ is unbounded,} \\
\lim\limits_{\lambda\in\Lambda}\mu_\lambda(A\cap X_\lambda), & \mbox{ if }
(\mu_\lambda(A\cap X_\lambda))_{\lambda\in\Lambda}\mbox{ is bounded}.
\end{cases}
\end{equation*}

The following result is Proposition~2.4 in \cite{KulkarniPamula1} and Proposition~3.7 in \cite{Gheondea4}. However, we 
slightly change the notation  to emphasise the fact that it consists of functions defined allover $X$. 

\begin{proposition}\label{p:wimu}
$\widetilde\mu$ is a measure which extends $\mu$ defined in \eqref{e:lmsa}. 
\end{proposition}

The following result is Lemma~3.8 in \cite{Gheondea4}.

\begin{lemma}\label{l:letexem}
Given a strictly inductive system of measure spaces 
$(X_\lambda,\Omega_\lambda,\mu_\lambda)_{\lambda\in\Lambda}$, let $(X,\Omega,\mu)$ denote its inductive limit as 
in \eqref{e:ilsm}. For each $\lambda\in\Lambda$,
consider the Hilbert space $L^2_\lambda(X,\mu)$ consisting of all functions $f\colon X\ra \CC$ such 
that $\supp(f)\subseteq X_\lambda$, $f|_{X_\lambda}$ is measurable with respect to 
$(X_\lambda,\Omega_\lambda)$, $\int_{X_\lambda} |f(x)|^2\de\mu_\lambda(x)<\infty$, and identified 
$\mu_\lambda$-a.e. Then, the net
$(L^2_\lambda(X,\mu))_{\lambda\in\Lambda}$ is a representing strictly inductive system
of Hilbert spaces. As a consequence, the space 
\begin{equation}\label{e:eltwoloc}
L^2_{\mathrm{loc}}(X,\mu):=\varinjlim_{\lambda\in\Lambda}L^2_\lambda(X,\mu),
\end{equation}
is the inductive limit of the strictly inductive system of Hilbert spaces
$(L^2_\lambda(X,\mu))_{\lambda\in\Lambda}$, and hence a representing locally Hilbert space. 
\end{lemma}

The representing locally Hilbert space $L^2_{\mathrm{loc}}(X,\mu)$ induced by the
inductive limit of a strictly inductive system of measure spaces 
$(X_\lambda,\Omega_\lambda,\mu_\lambda)_{\lambda\in\Lambda}$
is called \emph{concrete}. 
The definition of the Hilbert space $L^2_{\lambda}(X,\mu)$ as in Lemma~\ref{l:letexem} is not the 
usual one, but the space we define is canonically unitarily equivalent with the usual Hilbert space. This distinction 
is important and it should be made clear in order to avoid misunderstandings.
The next remark clarifies the structure of $L^2_{\mathrm{loc}}(X,\mu)$, 
cf.\ Remark~3.9 in \cite{Gheondea4}.
 
 \begin{remark}\label{r:eltwoloc} 
Let $f\in L^2_{\mathrm{loc}}(X,\mu)$ be arbitrary. Then, by definition,
there exists $\lambda_0\in\Lambda$ such that $f\in L^2_{\lambda_0}(X,\mu)$ in the sense that
$f$ is actually a function $\phi\colon X\ra \CC$ such that 
$f)\subseteq X_{\lambda_0}$, $f|_{X_{\lambda_0}}$ is $\Omega_{\lambda_0}$-measurable, 
$\int_{X_{\lambda_0}}|f(x)|^2\de\mu_{\lambda_0}<\infty$, and identified $\mu_{\lambda_0}$-a.e. 
Taking into account of the fashion in which the inclusion of the Hilbert 
spaces $L^2_\lambda(X,\mu)$ is defined, on the ground of assumptions (sim1)--(sim3),
$f$ is actually a class of functions $\mu$-a.e.\ on $X$. 
\end{remark}

\begin{remark}\label{r:linlvn} The Abelian locally $C^*$-algebra $L^\infty_{\mathrm{loc}}(X,\mu)$, see Remark~\ref{r:lef}, 
can be viewed as an Abelian locally von Neumann algebra on $L^2_{\mathrm{loc}}(X,\mu)$. 
Indeed, we identify each $\phi\in L^\infty_{\mathrm{loc}}(X,\mu)$ with 
the multiplication operator $M_\phi \colon L^2_{\mathrm{loc}}(X,\mu)\ra L^2_{\mathrm{loc}}(X,\mu)$ defined by 
$M_\phi f:=\phi f$, $f\in L^2_{\mathrm{loc}}(X,\mu)$ and, with notation as in Remark~\ref{r:lef}, $L^\infty_{\mathrm{loc}}(X,\mu)$
is the projective limit, in the category of locally convex $*$-algebras of the projective system of von Neumann algebras
$(L^\infty_\lambda(X,\mu))_{\lambda\in\Lambda}$ and the projective system of $*$-morphisms
$(\Phi_{\lambda,\nu})_{\lambda\leq \nu}$.
\end{remark}

As pointed out in Subsection~\ref{ss:lhs}, each locally Hilbert space $\cH$ is isometrically and densely 
embedded in its completion to a Hilbert space 
denoted by $\widetilde \cH$, where uniqueness is modulo a certain unitary operator, which 
is the inductive limit of the underlying strictly inductive system of Hilbert spaces within the category of Hilbert 
spaces with isometric morphisms.
By Proposition~3.10 in \cite{Gheondea4}, the above abstract completion can be made concrete for the case
of the locally Hilbert space $L^2_{\mathrm{loc}}(X,\mu)$: its completion 
can be realised as the Hilbert space $L^2(X,\widetilde\mu)$, where $\widetilde\mu$ is defined
as in \eqref{e:wimu}.

We need a result slightly more general than Lemma~\ref{l:letexem} in which we consider square integrable functions taking values in some fixed Hilbert space $\mathcal{G}$. The proof follows the same lines and we 
omit it.

\begin{proposition}\label{p:tensorprod} Let $\cG$ be a Hilbert space, 
$(X_\lambda,\Omega_\lambda,\mu_\lambda)_{\lambda\in\Lambda}$ be a strictly inductive system of measure 
spaces and let $(X,\Omega,\mu)$ be its inductive limit. If $\cG$ is a Hilbert space, for each 
$\lambda\in\Lambda$ we denote by $L^2_{\cG,\lambda}(X,\mu)$ the collection of all functions 
$f\colon X\ra \cG$ such that $\supp(f)\subseteq X_\lambda$, $f|_{X_\lambda}$ is measurable with respect to
$(X_\lambda,\Omega_\lambda)$, $\int_{X_\lambda} \|f(x)\|_\cG^2\de\mu_\lambda(x)<\infty$, and identified 
$\mu_\lambda$-a.e.

Then, the net $(L^2_{\cG,\lambda}(X,\mu))_{\lambda\in\Lambda}$ is a representing 
strictly inductive system of Hilbert spaces and, consequently, its inductive limit
\begin{equation}
L^2_{\cG,\mathrm{loc}}(X,\mu):=\varinjlim_{\lambda\in\Lambda} L_{\cG,\lambda}^2(X,\mu),
\end{equation}
exists and it is a representing locally Hilbert space.
\end{proposition}

\begin{remark}\label{r:tensorprod}  With notation as in Proposition~\ref{p:tensorprod}, for each 
$\lambda\in\Lambda$ there is a natural identification of the Hilbert spaces $L^2_{\cG,\lambda}(X,\mu)$ 
and $\cG\otimes L^2_\lambda(X,\mu)$. On elementary tensors this identification acts by 
$\cG\otimes L^2_\lambda(X,\mu)\ni g\otimes f\mapsto f g\in L^2_{\cG,\lambda}(X,\mu)$, where $(f \cdot g)x = f(x) g$, for all 
$x \in X$.
\end{remark}
\section{A Review of Direct Integrals of Hilbert Spaces}\label{s:rdihs}
Direct integrals of Hilbert spaces is an essential tool in the reduction theory of a von Neumann algebra by its 
centre, see the seminal paper of J.~von Neumann \cite{vonNeumann}.
There are a few more or less equivalent approaches to direct integrals of Hilbert Spaces, such as 
J.~Dixmier \cite{Dixmier}, A.~Ramsey \cite{Ramsey}, O.~Bratelli and D.W.~Robinson \cite{BratelliRobinson}, 
R.~Kadison and J.R. Ringrose \cite{KadisonRingrose}. 

\subsection{Borel Bundles of Hilbert Spaces}\label{ss:bbhs} 
Here we follow P.S.~Muhly \cite{Muhly}, which is essentially based 
on A.~Ramsey \cite{Ramsey}. Some definitions referring to topological spaces are as in S.M.~Srivastava \cite{Srivastava}. 
Let $X$ be a nonempty set and $\bH=\{\cH_x\}_{x\in X}$ be a family of Hilbert spaces 
indexed by $X$. We use the notation
\begin{equation*}
X\ast \bH :=\{(x,h)\mid x\in X,\ h\in \cH_x\}
\end{equation*}
and let $\pi\colon X\ast \bH\ra X$ denote the canonical projection $\pi(x,h)=x,$ for all $(x,h)\in X\ast \bH$. The pair 
$(X\ast \bH,\pi)$, or simply $X\ast \bH$, is called a \emph{Hilbert space bundle over $X$}. Technically, the vector 
bundle of Hilbert spaces $X\ast \bH$ is the \emph{disjoint  union} of the family of Hilbert spaces 
$\bH$, denoted by $\bigsqcup\limits_{x\in X}\cH_x$. This means that we identify the component $\cH_x$ of the 
disjoint union with the fiber $\{x\}\times \cH_x$ from $X\ast \bH$.

A \emph{section}, 
or a \emph{vector field}, of the bundle, 
denoted by $f\colon X\ra X\ast \bH$, is a right inverse of $\pi$, that is, $\pi\circ 
f=\mathrm{Id}_X$. This means that $f(x)\in \cH_x$ for all $x\in X$. Technically, the collection of all sections of the 
bundle $X\ast \bH$, denoted by $\cF_\bH(X)$, coincides with the product of the family of Hilbert spaces $\bH$, 
that is, $\prod\limits_{x\in X}\cH_x$, which is the collection of all maps $f\colon X\ra  \bigsqcup\limits_{x\in X}\cH_x$ such 
that $f(x)\in \cH_x$ for all $x\in X$. Here we identify $\cH_x$ with the fibre $\{x\}\times \cH_x$. This is important 
since $\cF_\bH(X)$ is naturally organised as a vector space, with respect to pointwise addition and multiplication 
with scalars. With respect to this, we identify $f\in \cF_\bH(X)$ with $(f(x))_{x\in X}$, where $f(x)\in \cH_x,$ for all 
$x\in X$.

In the following, a \emph{Borel space} is the same with a \emph{measurable space}, that is, a pair $(X,\Sigma)$, 
where $X$ is a nonempty set and $\Sigma$ is a $\sigma$-algebra on $X$. If $X$ is a Polish space (metrisable, 
complete, and separable) and $\Sigma$ is the Borel $\sigma$-algebra generated by the topology of $X$, then 
$(X,\Sigma)$ is called a \emph{standard space}. An \emph{analytic Borel space} is a countably generated Borel 
space which is the image of a standard Borel space under a Borel (measurable) mapping, e.g.\ see 
\cite{Srivastava}.

\begin{definition}\label{d:bb} 
A \emph{Borel bundle} is a Hilbert space bundle $(X\ast \bH,\pi)$, such that $(X\ast\bH,\Theta)$ 
is a Borel space and subject to the following conditions.
\begin{itemize}
\item[(1)] $(X,\Sigma)$ is a Borel space, where $\Sigma:=\{ E \subseteq X\mid \pi^{-1}(E)\in \Theta\}$.
\item[(2)] There exists a sequence $(f_n)_n$, with $f_n\in\cF_\bH(X)$ for all $n\in\NN$, called a \emph{fundamental 
sequence}, subject to the following conditions.
\begin{itemize}
\item[(a)] For each $n\in\NN$, the function $\tilde f_n\colon X\ast\bH\ra \CC$, defined by the formula 
$\tilde f_n(x,h):=\langle f_n(x),h\rangle_{\cH_x}$, is Borel (measurable).
\item[(b)] For each $m,n\in\NN$, the function $X\ni x\mapsto \langle f_m(x),f_n(x)\rangle_{\cH_x}$ is Borel.
\item[(c)] The set of functions $\{\tilde f_n\mid n\in\NN\}$ together with $\pi$, separate the points of the 
bundle $X\ast \bH$.
\end{itemize}
\end{itemize}
If, in addition,  the Borel space $(X\ast \bH,\Theta)$ is a standard (analytic Borel) space, then the Borel bundle 
$(X\ast \bH,\pi)$ is called \emph{standard} (\emph{analytic Borel}).
\end{definition}

\begin{remark}\label{r:div} 
Assume that $(X\ast \bH,\pi)$ is a Borel bundle as in Definition~\ref{d:bb}. 
\begin{enumerate}
\item[(i)] Condition 2.(c) means that, for any $x\in X$ and any distinct $h,k\in\cH_x$, there exists $n\in\NN$ such that 
$\langle f_n(x),h\rangle_{\cH_x}\neq \langle f_n(x),k\rangle_{\cH_x}$.

\item[(ii)] An arbitrary section $f\in\cF_\bH(X)$ is Borel (measurable) if and only if, for each $n\in\NN$, the map $X\ni x\mapsto \langle f(x),f_n(x)\rangle_{\cH_x}$ is Borel.

\item[(iii)] As a consequence of condition 2.(c), for each $x\in X$, the sequence of vectors $(f_n(x))_{n\in\NN}$ is total 
in $\cH_x$, in particular $\cH_x$ is separable.

\item[(iv)] Let 
\begin{equation*}
   B(X\ast \bH):=\left\{f\in\cF_\bH(X)\mid f\mbox{ is a Borel section}\right\}. 
\end{equation*}
The vector space $B(X\ast \bH)$ is a module over the space of Borel functions on $X$, that is  $B(X):=\{\phi\colon X\ra\CC\mid \phi\mbox{ Borel function}\}$.

\item[(v)] By means of the Gram-Schmidt orthonormalisation process, one can assume without loss of generality 
that the fundamental sequence $(f_n)_{n\in\NN}$ has the property that, for each $x\in X$, the sequence 
$(f_n(x))_{n\in\NN}$ is an orthonormal basis of $\cH_x$. Then, 
for $f\in B(X\ast\bH)$ and, for each $x\in X$, we have
\begin{equation*}
f(x)=\sum_{n=1}^\infty \alpha_n(x) f_n(x),
\end{equation*}
where $(\alpha_n(x))_{n\in\NN}\in \ell^2(\mathbb{N})$. By the Parseval identity, we have
\begin{equation*}
\|f(x)\|_{\cH_x}^2=\sum_{n=1}^\infty |\alpha_n(x)|^2,
\end{equation*}
in particular, since $\alpha_n(x)=\langle f(x),f_n(x)\rangle_{\cH_x}$, and hence $\alpha_n\colon X\ra\CC$ is Borel 
for all $n\in\NN$, it follows that the mapping $X\ni x\mapsto \|f(x)\|_{\cH_x}$ is Borel.
\end{enumerate}
\end{remark}
The following fact is useful.
\begin{proposition}[Proposition~3.2 in \cite{Muhly}]\label{p:bm} Let $X$ be an analytic Borel (standard) 
space, $\bH=\{\cH_x\}_{x\in X}$ a 
family of Hilbert spaces indexed by $X$, and let $(f_n)_{n\in\NN}$ be a sequence of sections of the bundle 
$X\ast \bH$ that satisfies the conditions 2.(b) and 2.(c) of Definition~\ref{d:bb}. Then, there exists a unique Borel 
structure $(X\ast \bH,\Theta)$ such that $(X\ast \bH,\pi)$ becomes an analytic Borel (standard) 
bundle such that $(f_n)_{n\in\NN}$ is a fundamental sequence of $X\ast \bH$.
\end{proposition}

\begin{remark}\label{r:unic} 
If $X$ is an analytic Borel space, by the uniqueness in Proposition~\ref{p:bm}, the definition of a 
Borel bundle as in
Definition~\ref{d:bb} does not depend on the fundamental sequence $(f_n)_{n\in\NN}$ of sections of the bundle 
$X\ast \bH$.
\end{remark}
The following result is a straightforward verification.
\begin{proposition}\label{p:di} 
If $(X\ast \bH,\pi)$ is a Borel bundle and $\mu$ is a measure on the Borel space 
$(X,\Sigma)$, then 
\begin{equation*}
\cL^2(X\ast \bH,\mu):=\{f\in B(X\ast\bH)\mid \int_X \|f(x)\|_{\cH_x}^2\de\mu(x)<\infty\},
\end{equation*}
is a pre-Hilbert space with respect to the
pre-inner product on $\cL^2(X\ast\bH,\mu)$, which is defined in the natural fashion,
\begin{equation*}
\langle f,g\rangle_{\cL^2(X\ast \bH,\mu)} :=\int_X \langle f(x),g(x)\rangle_{\cH_x}\de \mu(x),\quad f,g\in \cL^2(X\ast\bH),\mu).
\end{equation*}
\end{proposition}
\begin{definition}\label{d:di}
With notation and assumptions as in Proposition~\ref{p:di}, the quotient space
\begin{equation*}
 \int_X^\oplus \cH_x\de\mu(x):=\cL^2(X\ast\bH,\mu)/\sim,
\end{equation*}
where the equivalence relation $\sim$ 
is the identification of sections modulo subsets of $X$ of $\mu$-measure zero,
is called the \emph{direct integral Hilbert space} of the bundle $(X\ast \bH,\pi)$. Sometimes it is useful to use the 
simpler notation
\begin{equation}\label{e:ledoi}
L^2(X\ast\bH,\mu):=\int_X^\oplus \cH_x\de\mu(x).
\end{equation}
\end{definition}

For further use we make explicit the case of a standard Borel bundle which is ``trivial" in the vector bundle sense.

\begin{remark}\label{r:id} 
Let $(X,\Sigma,\mu)$ be a standard Borel measure space and $\cG$ a separable Hilbert space. Then the Hilbert space $\cG\otimes L^2(X,\mu)$ can be naturally identified with the direct integral Hilbert space 
$\int_X^\oplus \cG_x\de\mu(x)$, where $\cG_x$ is a copy of $\cG$, for all $x\in X$. To see this, we let 
$\bG=\{\cG_x\mid x\in X\}$ and consider the Hilbert space bundle $(X\ast \bG,\pi)$. In an equivalent 
formulation, $X\ast \bG=X\times \cG$, which is called a \emph{trivial} bundle over $X$.
Since $X$ is separable, let $X_0$ be a countable dense subset of it and, for each $p\in X_0$, 
by Urysohn Lemma we can get a continuous function $g_p\colon X\ra \CC$ with bounded support and 
$g_p(p)=1$. Since $\cG$ is 
separable, let $(h_m)_{m\in\NN}$ be an orthonormal basis of $\cG$ and, for each $m\in\NN$ and $p\in X_0,$ 
let $f_{p,m}\colon X\ra \cG$ be the function defined by $f_{p,m}(x)=g_p(x)h_m$, for all $x\in X$. The set
$\{f_{p,m}\mid p\in X_0,\ m\in\NN\}$ can be seen as a sequence since $X_0\times \NN$ is countable and we can 
check that this sequence satisfies the properties 2.(b) and 2.(c) from Definition~\ref{d:bb}. Indeed, for each 
$p,q\in X_0$ and each $m,k\in\NN$, for all $x\in X$ we have 
\begin{equation*}\langle f_{p,m}(x),f_{q,k}(x)\rangle_{\cH_x}= \langle g_p(x)h_m,g_q(x)h_k\rangle_{\cH_x}
= g_p(x)\overline{g_q(x)}\langle h_m,h_k\rangle_{\cH_x}=\delta_{m,k} g_p(x)\overline{g_q(x)},
\end{equation*}
hence condition 2.(b) from Definition~\ref{d:bb} is satisfied. Also, for arbitrary but fixed 
$x\in X$, let $h,k\in\cG$ be such that $\langle f_{p,m}(x),h\rangle_{\cH_x}=\langle f_{p,m}(x),k\rangle_{\cH_x}$ for 
all $p\in X_0$ and $m\in\NN$, that is,
\begin{equation*}
g_p(x)\langle h_m,h-k\rangle_{\cH_x}=0,\quad p\in X_0,\ m\in\NN.
\end{equation*}
Since $X_0$ is dense in $X$, there exists $p\in X_0$ such that $g_p(x)\neq 0$ and then, since 
$(h_m)_{m\in\NN}$ is total in $\cG$, it follows that $h=k$. In view of Remark~\ref{r:div}, 
this means that condition 2.(c) in Definition~\ref{d:bb} is satisfied as well.

Then,
by Proposition~\ref{p:bm} it follows that there exists a unique standard Borel structure 
$(X\ast\bG,\Theta)$ such that $(X\ast\bG,\pi)$ becomes a standard Borel bundle with respect to 
which $\{f_{p,m}\mid p\in X_0,\ m\in\NN\}$ is a fundamental sequence. This implies that we can define the 
direct integral Hilbert space $\int_X^\oplus \cG_x\de\mu(x)$ as in Proposition~\ref{p:di} and Definition~\ref{d:di}.
Because all $\cG_x$ are copies of the same Hilbert space $\cG$, it follows that a section $f\in\cF_\bG(X)$ is 
nothing else but a function $f\colon X\ra \cG$. By fixing an orthonormal basis $(h_m)_{m\in\NN}$ of $\cG$, 
this means that
for any $f\in\cF_\bG(X)$ there exists uniquely a sequence $(\alpha_m)_{m\in\NN}$ of functions 
$\alpha_m\colon X\ra \CC$ such that
\begin{equation*}
f(x)=\sum_{m\in\NN} \alpha_m(x) h_m,\quad x\in X,
\end{equation*}
with the property that, for $\mu$-a.e.\ $x\in X$, we have $(\alpha_m(x))_{m\in\NN}\in \ell^2(\mathbb{N})$. Since, 
actually $\alpha_m(x)=\langle f(x),h_m\rangle_\cG$ $\mu$-a.e.\ $x\in X$ and all $m\in\NN$, it follows that 
$f\in B(X\ast \bG)$ if and only if $\alpha_m$ is Borel for all $m\in\NN$ and hence, with notation as in 
Proposition~\ref{p:di}, $f\in\cL^2(X\ast \bG,\mu)$ if and only if $\alpha_m$ is Borel for all $m\in\NN$ and 
\begin{equation*}
\sum_{m\in\NN}\int_X |\alpha_m(x)|^2\; \de\mu(x)=\int_X \sum_{m\in\NN} |\alpha_m(x)|^2\;\de\mu(x)=
\int_X \|f(x)\|^2_{\cG}\; \de\mu(x)<\infty.
\end{equation*} 

Since $X$ is a standard Borel space, the Hilbert space $L^2(X,\mu)$ is separable and let $(\phi_m)_{m\in\NN}$ 
be an orthonormal basis of it. Then $\{h_m\otimes \phi_k\mid m,k\in\NN\}$ is an orthonormal basis of 
$\cG\otimes L^2(X,\mu)$. We observe that, for each $m,k\in\NN$ we have 
$\phi_k h_m\colon X\ra \cG$ and, actually, $\phi_k h_m\in L^2(X\ast \cG,\mu)$. Since, in view of the 
Definition~\ref{d:di} and notation as in Proposition~\ref{p:di}, we have
\begin{align*}
\langle \phi_k h_m,\phi_l h_n\rangle_{L^2(X\ast \bG, \mu)} & =\int_X\langle \phi_k(x) h_m,\phi_l(x) h_n\rangle_\cG\de\mu(x) \\
& = \int_X \phi_k(x)\ol{\phi_l(x)}\de\mu(x) \langle h_m,h_n\rangle_\cG
=\delta_{k,l}\delta_{m,n},
\end{align*}
for $m,n,k,l\in\NN$, it follows that $\{ \phi_k h_m\mid k,m\in\NN\}$ is an orthonormal basis of $L^2(X\ast \bG,\mu)$. The identification 
of these orthonormal bases provides the required unitary identification of the Hilbert spaces 
$\cG\otimes L^2(X,\mu)$ with $L^2(X\ast \bG,\mu)$. Further on, we identify the Hilbert spaces 
$\cG\otimes L^2(X,\mu)$ and $L^2_\cG(X,\mu)$ as in Remark~\ref{r:tensorprod}.
\end{remark}

\subsection{Decomposable and Diagonalisable Operators} \label{ss;ddo}
We continue with assumptions and notation as in the previous subsection. So, 
$L^2(X\ast\bH,\mu)=\int_X^\oplus \cH_x\de\mu(x)$ is the 
direct integral Hilbert space associated to an analytic Borel bundle $(X\ast \bH,\pi)$ whose elements 
are sections $f\in B(X\ast\bH,\mu)$, identified $\mu$-a.e., such that $\int_X \|f(x)\|_{\cH_x}^2\de\mu(x)<\infty$. 
Let $T\colon L^2(X\ast\bH,\mu)\ra L^2(X\ast \bH,\mu)$ be a 
bounded linear operator. We call $T$ \emph{decomposable} if there exists a family $\{T_x\}_{x\in X}$, where 
$T_x\in \cB(\cH_x)$, for all $x\in X$, such that, for all $f\in L^2(X\ast\bH,\mu)$
 we have 
\begin{equation*}
(Tf)(x)=T_x f(x),\quad \mu\mbox{-a.e. } x\in X.
\end{equation*}
Also, a decomposable operator $T$ is called \emph{diagonalisable} if there exists $\Phi\in L^\infty(X,\mu)$, such 
that $T_x$ is the multiplication operator with $\Phi(x)$ $\mu$-a.e.\ $x\in X$, that is,
for all $f\in L^2(X\ast\bH,\mu)$, we have
\begin{equation*}
(Tf)(x)=\Phi(x) f(x),\quad  \mu\mbox{-a.e. } x\in X.
\end{equation*}

We will use the following theorem, e.g.\ see \cite{Dixmier}, 
in Theorem \ref{t:locdecdeg}.

\begin{theorem} \label{t:decdeg}
Let 
\begin{equation*}\cB_{\mathrm{dec}}(L^2(X\ast\bH,\mu)):=\{T\in \cB(L^2(X\ast\bH,\mu))\mid T\mbox{ decomposable}\},\end{equation*} 
and
\begin{equation*}
\cB_{\mathrm{diag}}(L^2(X\ast\bH,\mu)):=\{T\in \cB(L^2(X\ast\bH,\mu))\mid T\mbox{ diagonalisable}\}.
 \end{equation*}
Then $\cB_{\mathrm{dec}}(L^2(X\ast\bH,\mu))$ and $\cB_{\mathrm{diag}}(L^2(X\ast\bH,\mu))$ are von Neumann algebras, 
the latter is Abelian, and each of it is the commutant of the other.
\end{theorem}

For later reference, we make explicit the transfer of an amplification of an $L^\infty$ algebra to an algebra of 
diagonalisable operators.

\begin{remark}\label{r:ampdiag}
With notation as in Remark~\ref{r:id}, consider the Abelian $W^*$-algebra $L^\infty(X,\mu)$ and 
its $\cG$-amplification $I_{\cG}\otimes L^\infty(X,\mu)$, where $I_\cG$ denotes 
the unit operator on the Hilbert space $\cG$. Because the Hilbert 
space $\cG$ is supposed to be separable, without loss of generality we can assume that $\cG=\ell^2_n$, with 
$n\in\NN\cup\{\infty\}$. Then, the 
elements of $I_n\otimes L^\infty(X,\mu)$ are matrices $I_n\otimes \Phi$, that is, diagonal $n\times n$-matrices 
with all entries on the diagonal a function $\Phi\in L^\infty(X,\mu)$. There is a natural action of the $W^*$-algebra 
$I_n\otimes L^\infty(X,\mu)$ on the Hilbert space $\ell^2_n \otimes L^2(X,\mu)$: on elementary tensors 
$\xi\otimes f$, with $\xi\in \ell^2_n$ and $f\in L^2(X,\mu)$, its action is $(I_n\otimes \Phi)(\xi\otimes f)=
\xi \otimes \Phi f$. In this way, we view the $W^*$-algebra $I_n\otimes L^\infty(X,\mu)$ as a von Neumann 
algebra living in $\cB(\ell^2_n\otimes L^2(X,\mu))$.

As in Remark~\ref{r:id}, we consider the trivial bundle $X\ast \bG=X\times \ell^2_n$, where the bundle 
$\bG=\{\cG_x\}_{x\in X}$ with $\cG_x$ a copy of $\ell^2_n$ for all $x\in X$. Then, we can write simply 
$L^2(X\times \ell^2_n,\mu)=L^2(X\ast \bG,\mu)$. Let 
\begin{equation*}U\colon \ell^2_n\otimes L^2(X,\mu)\ra L^2(X\times \ell^2_n,\mu)
=\int_X^\oplus \ell^2_n\de\mu(x)\end{equation*}
denote the unitary 
operator defined on the orthonormal basis $\{h_m\otimes \phi_k\mid 1\leq m\leq n,\; k\in\NN\}$ by
\begin{equation}\label{e:uhamo}
U(h_m\otimes \phi_k)=\phi_k h_m\in L^2(X\times\ell^2_n,\mu)
\end{equation} and then extended on Fourier series. 
This unitary operator $U$ implements the spatial $*$-isomorphism of the Abelian von Neumann algebra 
$I_n\otimes L^\infty(X,\mu)$ with $\cB_{\mathrm{diag}}(L^2(X\times\ell^2_n,\mu))$, the Abelian 
von Neumann algebra of all diagonalisable operators on $L^2(X\times \ell^2_n,\mu)$. To see this, let 
$\Psi\in L^\infty(X,\mu)$ and $1\leq m \leq n,\; k \in\NN$ be arbitrary, and note that
\begin{equation}\label{e:tui} 
T(\phi_k h_m)=(U (I_n\otimes \Psi)U^*)(\phi_k h_m)= U(I_n\otimes \Psi)(h_m\otimes \phi_k)=U(h_m \otimes \Psi \phi_k,)
\end{equation}
hence, in view of the definition of $U$ as in \eqref{e:uhamo}, it follows that each fibre $\cG_x=\ell^2_n$ is left 
invariant under $T$ hence, denoting the compression of $T$ at $\cG_x$ 
by $T_x\colon \cG_x\ra \cG_x$, for $x\in X$, it follows that 
$(Tf)(x)=T_x f(x)$ for arbitrary $f\in L^2(X\times \ell^2_n,\mu)$ and $x\in X$, 
which proves that $T$ is decomposable. Finally, from \eqref{e:tui} it 
follows that, for arbitrary $x\in X$, the action of $T_x$ is just the multiplication with $\Psi(x)$, hence $T$ 
is diagonalisable.
\end{remark}

\subsection{Reduction of Abelian von Neumann Algebras}\label{ss:ravna}
A very important fact in our enterprise is the following result, see
Corollary III.1.5.18 in \cite{Blackadar}. The additional fact corresponding to the case when the Hilbert space 
$\cH$ is separable can be seen in Subsection 4.4.1 in \cite{BratelliRobinson}.

\begin{theorem}\label{t:abvn}
Let $\mathcal{Z}$ 
be an Abelian von Neumann algebra on the Hilbert space $\mathcal{H}$. Then, there exists a 
unique decomposition
\begin{equation*}
\mathcal{H}=\mathcal{H}_\infty\oplus\bigoplus_{n\geq 1} \mathcal{H}_n,
\end{equation*}
such that, for each $n\in \mathbb{N}\cup \{\infty\}$, $\mathcal{H}_n$ is invariant under $\mathcal{Z}$ and there 
exists a locally finite measure space $(X_n,\Omega_n,\mu_n)$ such that the 
compression of $\mathcal{Z}$ to $\mathcal{H}_n$, denoted by  $\mathcal{Z}_{\mathcal{H}_n}:=
P_{\cH_n}\mathcal{Z}|_{\cH_n}$, when viewed as a von Neumann algebra in $\cB(\cH_n)$,
is spatially isomorphic to the $n$-fold amplification 
$I_n\otimes L^\infty(X_n,\mu_n)$ acting on $\ell^2_n\otimes L^2(X_n,\mu_n)$, that is, there exists a unitary 
operator 
\begin{equation*}U_n\colon \cH_n\ra \ell^2_n\otimes L^2(X_n,\mu_n),\end{equation*} 
such that the map
\begin{equation*}
\mathcal{Z}_{\mathcal{H}_n}\ni T\mapsto  U_n TU_n^*\in I_n\otimes L^\infty(X_n,\mu_n)
\end{equation*}
is an isomorphism of von Neumann algebras. If the Hilbert space $\cH$ is separable, then the measure spaces
$(X_n,\Omega_n,\mu_n)$, for all $n\in\NN\cup\{\infty\}$, can be found standard. 
\end{theorem}

In case the Hilbert space $\cH$ is separable, one can transfer the previous theorem into a representation of the 
von Neumann algebra $\cZ$ to a certain Abelian von Neumann algebra of all diagonalisable operators 
with respect to a certain direct integral representation of 
$\cH$. 

\begin{theorem}\label{t:did}
Let $\mathcal{Z}$ be an Abelian von Neumann algebra on the separable Hilbert space $\mathcal{H}$. Then,
there exists a standard Borel bundle $(X\ast \bH,\pi)$, a measure $\mu$ on $X$, and a unitary 
operator $U\colon \cH\ra L^2(X\ast \bH,\mu)$ such that, with notation as in \eqref{e:ledoi}, we have the spatial 
isomorphism of von Neumann algebras
\begin{equation*}
U\cZ U^* =\cB_{\mathrm{diag}}(L^2(X\ast\bH,\mu)).
\end{equation*}
\end{theorem}

This is because, in case $\cH$ is separable, hence $X_n$ obtained in Theorem \ref{t:abvn} are standard spaces and, consequently,
$L^2(X_n,\mu_n)$ are separable, 
for each $n\in \NN\cup\{\infty\}$, see Remark~\ref{r:id}, the Hilbert space $\ell^2_n\otimes L^2(X_n,\mu_n)$ is 
unitarily equivalent with the direct integral Hilbert space $\int_{X_n}^\oplus \cH_x\de\mu_n(x)
=L^2(X_n\ast \bH_n,\mu_n)$, where $\bH_n=\{\cH_x\}_{x\in X_n}$ with $\cH_x=\ell^2_n$ for all $x\in X_n$.
Then, letting $X=X_\infty\sqcup\bigsqcup_{n\geq 1} X_n$, we can organise it in a natural fashion as a standard 
Borel space and then, letting $\bH=\bH_\infty\sqcup\bigsqcup_{n\geq 1}\bH_n$, it follows 
that the von Neumann algebra $\cZ$ is spatially isomorphic with the von Neumann algebra 
$\cB_{\mathrm{diag}}(L^2(X\ast\bH,\mu))$
of all diagonalisable 
operators with respect to the standard Borel bundle $(X\ast\bH,\pi)$.

\section{Direct Integrals of Locally Hilbert Spaces}\label{s:dilhs}
In this section, we define the notion of direct integral of locally Hilbert spaces and present the 
reduction of Abelian locally von Neumann algebra. We use notations and terminology from the previous sections.

\subsection{Borel Bundles of Locally Hilbert Spaces}\label{ss:bblhs}
For the beginning, we set the stage for the notion of a direct integral of locally Hilbert spaces and explain what is its purpose and 
what we want to get.
Assume that, for a given directed set $\Lambda$, we have a strictly inductive system of sets
$(X_\lambda)_{\lambda\in\Lambda}$ in the sense that, whenever $\lambda,\nu\in \Lambda$ are such that $\lambda\leq \nu$,
we have $X_\lambda\subseteq X_\nu$. Let 
$X=\varinjlim\limits_{\lambda\in\Lambda} X_\lambda=\bigcup\limits_{\lambda\in\Lambda} X_\lambda$ be its inductive 
limit. Also, assume that, for each $x\in X$, 
there is given a net $(\cH_{x,\lambda})_{\lambda\in \Lambda}$, which is a 
strictly inductive system of Hilbert spaces, and let $\cH_x=\varinjlim\limits_{\lambda\in\Lambda} \cH_{x,\lambda}$ be 
the associated locally Hilbert space. We consider the bundle of locally Hilbert spaces 
$\bH=\{\cH_x\}_{x\in X}$. 

For each $\lambda \in \Lambda$, we consider 
$\bH_\lambda=\{\cH_{x,\lambda}\}_{x\in X_\lambda}$ and the bundle of Hilbert spaces 
$(X_\lambda\ast \bH_\lambda,\pi_\lambda)$, see 
Subsection~\ref{ss:bbhs}, and clearly, whenever $\lambda,\nu\in \Lambda$ are such that $\lambda\leq \nu$ we 
have $X_\lambda\ast\bH_\lambda\subseteq X_\nu\ast \bH_\nu$. This shows that
$(X_\lambda\ast\bH_\lambda)_{\lambda\in\Lambda}$ is a strictly inductive system of bundles of Hilbert spaces. 
Then, with notation as before, we denote
\begin{equation*}
\varinjlim_{\lambda\in\Lambda} (X_\lambda\ast \bH_\lambda)=
X\ast\bH=\bigcup_{\lambda\in\Lambda} (X_\lambda \ast \bH_\lambda),
\end{equation*}
and call it the \emph{inductive limit} of the net of bundles of Hilbert spaces 
 $(X_\lambda\ast\bH_\lambda)_{\lambda\in \Lambda}$.
Note that the bundle  $X\ast \bH$ has the canonical projection $\pi\colon X\ast\bH\ra X$ defined in the usual 
fashion: $\pi(x,h)=x,$ for all $(x,h)\in X\ast \bH$. 
We refer to the pair $(X\ast\bH,\pi)$ as a \emph{bundle of locally Hilbert spaces}.

As usually, a \emph{section}, or a \emph{vector field}, of the bundle $(X\ast \bH,\pi)$ is a map 
$f\colon X\ra X\ast \bH$ which is a right inverse of $\pi$, that is $\pi\circ f=\mathrm{Id}_X$, meaning that, 
for each $x\in X$ there exists 
$\lambda\in\Lambda$ such that $x\in X_\lambda$ and $f(x)=(x,h)$ with $h\in\cH_{x,\lambda}$. Because of 
this, we call $f$ a \emph{coherent section} of $X\ast \bH$.
Let $\cF_{\bH}(X)$ denote the collection of all coherent sections $f$ of the bundle
of locally Hilbert spaces $(X\ast \bH,\pi)$. Clearly, $\cF_{\bH}(X)$ is a vector space.

In addition to this setting, let us now assume that $(X_\lambda,\Sigma_\lambda,\mu_\lambda)_{\lambda\in\Lambda}$ 
is a strictly inductive net of Borel measure spaces and let 
$(X,\Sigma,\mu)=\varinjlim\limits_{\lambda\in\Lambda} (X_\lambda,
\Sigma_\lambda,\mu_\lambda)$ be its inductive limit. Also, assume that, for each $x\in X$, 
there is given a net $(\cH_{x,\lambda})_{\lambda\in \Lambda}$, which  is a 
strictly inductive system of Hilbert spaces, and let $\cH_x=\varinjlim\limits_{\lambda\in\Lambda} \cH_{x,\lambda}$ be 
the associated locally Hilbert space. Moreover, assume that
we have a net of Borel bundles of Hilbert spaces $(X_\lambda\ast\bH_\lambda,\Theta_\lambda,
\pi_\lambda)$, in the sense of Definition~\ref{d:bb}, 
which is a strictly inductive system of Borel spaces, and let $(X\ast \bH,\Theta,
\pi)=\varinjlim\limits_{\lambda\in\Lambda} (X_\lambda\ast \bH_\lambda,\Theta_\lambda,\pi_\lambda)$, in the sense 
of \eqref{e:lms}. To be more precise, we have
\begin{equation}\label{e:lmsa1}
X\ast\bH=\bigcup_{\lambda\in\Lambda} X_\lambda\ast\bH_\lambda,\quad \Theta=\bigcup_{\lambda\in\Lambda}
\Theta_\lambda.
\end{equation}

With notation as in Proposition~\ref{p:di}, for each $\lambda\in\Lambda$, consider the vector space
\begin{align}\label{e:celt}
\cL^2_\lambda(X\ast\bH,\mu) & :=\big\{ f\in\cF_\bH(X)\mid \supp(f)\subseteq X_\lambda,\ f|_{X_\lambda}\in 
B(X_\lambda\ast \bH_\lambda,\Theta_\lambda),\\ 
& \phantom{\supp(f)\subseteq X_\lambda\ \ \ \ \ } \int_{X_\lambda}\|f(x)\|^2_{\cH_{x,\lambda}}\de \mu_\lambda(x)<\infty\big\},\nonumber
\end{align}
on which we consider the pre-inner product
\begin{equation}\label{e:lafag}
\langle f,g\rangle_{\cL^2_\lambda(X\ast\bH,\mu)}:=\int_{X_\lambda} \langle f(x),g(x)\rangle_{\cH_{x,\lambda}} 
\de\mu_\lambda(x),\quad f,g\in \cL^2_\lambda(X\ast\bH,\mu).
\end{equation}
Letting $\sim$ denote the equivalence relation on $\cL^2_\lambda(X\ast\bH,\mu)$ 
defined by identification modulo subsets of $X_\lambda$ of $\mu_\lambda$-measure zero, define
\begin{equation}
L^2_\lambda(X\ast\bH,\mu):=\cL^2_\lambda(X\ast\bH,\mu)/\sim, 
\end{equation}
with inner product defined as in \eqref{e:lafag}, and note that it is a Hilbert space. The Hilbert spaces
$L^2_\lambda(X\ast\bH,\mu)$ make the prototypes of direct integral Hilbert spaces that we need in order to 
define the concept of direct integral of locally Hilbert spaces because, for all $\lambda$ in $\Lambda$, their 
elements are functions (actually, equivalence classes of functions) defined on the whole set $X$. Although they 
are not quite the type of a direct integral Hilbert space that stems from Definition~\ref{d:di}, they are canonically 
isomorphic with them. More precisely, consider the direct integral Hilbert space 
$L^2(X_\lambda\ast\bH_\lambda,\mu_\lambda)$ defined as in Definition~\ref{d:di}, with $\lambda\in\Lambda$ 
fixed. Then, the map $\cL^2_\lambda(X\ast\bH,\mu)\ni f\mapsto f|_{X_\lambda}\in 
\cL^2(X_\lambda\ast\bH_\lambda,\mu_\lambda)$ is uniquely lifted to an isomorphism of Hilbert spaces 
$L^2_\lambda(X\ast \bH,\mu)\ra L^2(X_\lambda\ast\bH_\lambda,\mu_\lambda)$.

After all these considerations, we can now see that the crucial point is to impose the condition that the net of Hilbert spaces
$(L^2_\lambda(X\ast\bH,\mu))_{\lambda\in\Lambda}$ is strictly inductive. So, what we have to do is to define what a Borel 
bundle of locally Hilbert spaces is and then what is the corresponding direct integral of locally Hilbert spaces.
 
\begin{definition}\label{d:bblhs} 
Let $\big ( \Lambda, \leq \big)$ be a directed set and let $\bigl(X_{\lambda}, \Sigma_{\lambda}, 
\mu_{\lambda} \bigr)_{\lambda \in \Lambda}$ be a strictly inductive system of Borel measure spaces with 
$(X, \Sigma, \mu) := \varinjlim\limits_{\lambda \in \Lambda} (X_{\lambda}, \Sigma_{\lambda}, \mu_{\lambda})$ its inductive limit. 
For each $x \in X$, let $\big \{ \mathcal{H}_{x, \lambda} \big \}_{\lambda \in \Lambda}$ be a strictly inductive system of separable 
Hilbert 
spaces and $\mathcal{H}_x := \varinjlim\limits_{\lambda \in \Lambda} \mathcal{H}_{x, \lambda}$ be the associated locally Hilbert 
space. For each $\lambda\in\Lambda$, denote $\bH_\lambda:=\{\cH_{x,\lambda}\}_{x\in X_\lambda}$. We call 
\begin{equation*}\bigl( X \ast \mathbf{H}, \pi \bigr) = \varinjlim\limits_{\lambda \in \Lambda}
 \bigl( X_\lambda \ast \mathbf{H}_\lambda, \pi_\lambda \bigr)\end{equation*} 
a \emph{Borel bundle of locally Hilbert spaces}, if the following conditions hold.
\begin{enumerate}
\item[(i)] There exists a net of Borel bundles of Hilbert spaces 
$\big ( X_\lambda \ast \mathbf{H}_\lambda, \Theta_\lambda, \pi_\lambda \big)_{\lambda \in \Lambda}$, which is a strictly 
inductive system of Borel spaces, and let, see \eqref{e:lmsa1}, 
\begin{equation*}\bigl( X \ast \mathbf{H}, \Theta, \pi \bigr) = \varinjlim\limits_{\lambda \in \Lambda} 
\bigl( X_\lambda \ast \mathbf{H}_\lambda, \Theta_\lambda, \pi_\lambda \bigr).\end{equation*}
\item[(ii)] $\Sigma_\lambda = \big \{ E \subseteq X_\lambda \mid \pi^{-1}_\lambda(E) \in \Theta_\lambda \big \}$, for all 
$\lambda \in \Lambda$.
\item[(iii)] The net $\big ( L^2_\lambda \big (X \ast  \mathbf{H}, \mu \big ) \big)_{\lambda \in \Lambda}$, see \eqref{e:celt}, which exists as a consequence of (i), forms a strictly inductive system of Hilbert spaces.
\end{enumerate}
The locally Hilbert space defined by
\begin{equation}\label{e:leloc}
L^2_{\mathrm{loc}}(X\ast\bH,\mu)=\varinjlim_{\lambda\in \Lambda} L^2_\lambda(X\ast \bH,\mu),
\end{equation}
which exists in view of assumption (iii), see \eqref{e:lmsa1} through \eqref{e:lafag},
is called the \emph{direct integral of locally Hilbert spaces} of the bundle $(X\ast \bH,\pi)$, where 
$(X,\Sigma,\mu):=\varinjlim_{\lambda\in\Lambda}(X_\lambda,\Sigma_\lambda,\mu_\lambda)$ and  
$\bH:=\{\cH_x\}_{x\in X}$. Also, we can use the more involved notation
\begin{equation}\label{e:intexo}
\int_X^{\oplus,\mathrm{loc}} \cH_x\de \mu(x):=L^2_{\mathrm{loc}}(X\ast \bH,\mu).
\end{equation}

If, in addition, for all $\lambda\in\Lambda$, the Borel space $(X_\lambda,\bH_\lambda,\Theta_\lambda)$ 
is standard (analytic Borel) then the Borel bundle of locally Hilbert spaces $(X\ast \bH,\pi)$ is called 
\emph{standard} (\emph{analytic Borel}).
 \end{definition}

Let us observe that, under the assumptions of Definition~\ref{d:bblhs}, the elements of the direct integral
of locally Hilbert spaces $L^2_{\mathrm{loc}}(X\ast \bH,\mu)$ are coherent sections, identified $\mu$-a.e.

\begin{remark}\label{r:dep}
One may note that the direct integral of locally Hilbert spaces of the bundle $\big ( X \ast \mathbf{H}, \pi \big)$, denoted by 
$L^2_{\mathrm{loc}}(X \ast \mathbf{H}, \mu)$, depends on the choice of the bundle of topologies 
$\{ \Theta_\lambda\}_{\lambda \in \Lambda}$. 
However, we will not highlight this fact in the notation of the direct integral of locally Hilbert spaces of the bundle 
$\big ( X \ast \mathbf{H}, \pi \big)$. We will mention it only when it becomes necessary.
\end{remark}

\begin{lemma}\label{l:sigae} With notation and assumptions as in Definition~\ref{d:bblhs}, we have
\begin{equation*}\Sigma=\{E\subseteq X\mid \pi^{-1}(E)\in \Theta\}.\end{equation*}
\end{lemma}

\begin{proof}
Indeed, if $E\in\Sigma$, then there exists $\lambda\in\Lambda$ such that $E\subseteq \Sigma_\lambda$ and then, by (ii),
we have $\pi_\lambda^{-1}(E)\in\Theta_\lambda$. Since $E\subseteq X_\lambda$, we have 
$\pi^{-1}(E)=\pi_\lambda^{-1}(E)\in \Theta_\lambda\subseteq \Theta$, hence $\pi^{-1}(E)\in \Theta$.

Conversely, let $E\subseteq X$ be such that $\pi^{-1}(E)\in\Theta=\bigcup_{\lambda\in\Lambda}\Theta_\lambda$, hence
there exists $\lambda\in\Lambda$ such that $\pi^{-1}(E)\in\Theta_\lambda$. But, $E\subseteq X=\bigcup_{\mu\in\Lambda} X_\mu$
implies that there exists $\mu\in\Lambda$ such that $E\subseteq X_\mu$. Since $\Lambda$ is directed, there exists 
$\nu\in\Lambda$ such that $\mu\leq\nu$ and $\lambda\leq\nu$, hence $E\subseteq X_\nu$ and 
$\pi^{-1}(E)\in\Theta_\lambda\subseteq \Theta_\nu$.
\end{proof}

After getting the concept of direct integral of locally Hilbert spaces, which is a locally Hilbert space, the first natural question to be 
answered is whether it is a representing (commutative) one.

\begin{proposition}\label{p:dintrep}
Consider the notation and assumptions as in Definition~\ref{d:bblhs}.

\nr{1} For each $\lambda \in\Lambda$, the linear operator 
\begin{equation}\label{e:pela}
L^2_{\mathrm{loc}}(X \ast \mathbf{H}, \mu)\ni 
f\mapsto  \chi_{X_\lambda}f\in L^2_\lambda(X \ast \mathbf{H}, \mu)
\end{equation} is the Hermitian projection $P_\lambda$ from 
$ L^2_{\mathrm{loc}}(X \ast \mathbf{H}, \mu)$ onto $L^2_\lambda(X \ast \mathbf{H}, \mu)$.

\nr{2} The locally Hilbert space $L^2_{\mathrm{loc}}(X \ast \mathbf{H}, \mu)$ is representing.
\end{proposition}

\begin{proof} (1) It is easy to see that the operator defined at \eqref{e:pela} is a Hermitian projection with range 
$L^2_\lambda(X \ast \mathbf{H}, \mu)$ and, by uniqueness, it coincides with $P_\lambda$.

(2)  This is a direct consequence of the statement at item (1).
\end{proof}
 
Any locally Hilbert space has completion to a Hilbert space but, in general, this is an abstract completion.
The second question refers to the possibility to find a concrete form of such a Hilbert space completion of a direct integral locally 
Hilbert space. The answer is positive under the additional assumption that the directed set $\Lambda$ is countable and then
one can always find the Hilbert space completion a direct integral of Hilbert spaces.
 
 \begin{theorem}\label{t:dintcomp} With notation and assumptions as in Definition~\ref{d:bblhs}, consider the measure space
 $(X,\widetilde \Sigma,\widetilde \mu)$, constructed as in \eqref{e:tomega} and \eqref{e:wimu}. Similarly, let $\widetilde\Theta$
 be the $\sigma$-algebra constructed from $\Theta$ as in \eqref{e:tomega}. For each $x\in X$, let $\widetilde \cH_x$ be a Hilbert 
 space completion of the locally Hilbert space $\cH_x$ and consider the Hilbert space bundle 
 $\widetilde\bH=\{\widetilde\cH_x\}_{x\in X}$. In addition, assume that the directed set $\Lambda$ is countable.
 The following statements hold true.
 
 \nr{1} $(X\ast\widetilde \bH,\widetilde \Theta)$ is a Borel bundle of Hilbert spaces, in the sense of Definition~\ref{d:bb},
 and $\widetilde\Sigma=\{E\in X\mid \pi^{-1}(E)\in\widetilde\Theta\}$.
 
 \nr{2} The direct integral Hilbert space $L^2(X\ast\widetilde\bH,\widetilde\mu)$, which exists due to the statement at item (1),
 contains densely the direct integral locally Hilbert space $L^2_{\mathrm{loc}}(X\ast\bH,\mu)$.
 \end{theorem}
 
 \begin{proof} (1) The equality $\widetilde\Sigma=\{E\subseteq X\mid \pi^{-1}(E)\in\widetilde\Theta\}$
 is a direct consequence of Lemma~\ref{l:sigae} and the construction of $\widetilde\Sigma$ and
 $\widetilde\Theta$ as in \eqref{e:tomega}.
 
By assumption (i) in Definition~\ref{d:bblhs}, for each 
$\lambda\in\Lambda$, $\big ( X_\lambda \ast \mathbf{H}_\lambda, \Theta_\lambda, \pi_\lambda \big)$ is a Borel bundle of Hilbert 
spaces and then, in view of Definition~\ref{d:bb}, there exists a countable set 
$\cD_\lambda\subseteq \cF_{\bH_\lambda}(X_\lambda)$ subject to the following conditions.
\begin{itemize}
\item[(a)] For each $f\in \cD_\lambda$, the map $\tilde f\colon X_\lambda\ast\bH_\lambda\ra \CC$, 
defined by $\tilde f(x,h):=\langle f(x),h\rangle_{\cH_{x,\lambda}}$, is Borel.
\item[(b)] For each $f,g\in \cD_\lambda$, the map $X_\lambda \ni x\mapsto \langle f(x),g(x)\rangle_{\cH_{x,\lambda}}$ is Borel.
\item[(c)]  The set of functions $\{ \tilde f\mid f\in\cD_\lambda\}\cup \{\pi_\lambda\}$ separates the points of the bundle
$X_\lambda\ast\bH_\lambda$.
\end{itemize}

Since $\Lambda$ is countable, the set $\cD:=\bigcup_{\lambda\in\Lambda}\cD_\lambda$ is countable and, by considering trivial 
extensions  of its elements from $X_\lambda$ to $X$, it can be viewed as a subset of $\cF_{\bH}(X)$, in particular a subset
of $\cF_{\widetilde \bH}(X)$. From here and the properties (a) through (c) from above, it is easy to see that $\cD$ 
has the following properties, with respect to $\Sigma$ and $\Theta$.
\begin{itemize}
\item[($\alpha$)] For each $f\in \cD$, the map $\tilde f\colon X\ast\bH\ra \CC$, 
defined by $\tilde f(x,h):=\langle f(x),h\rangle_{\cH_{x}}$, is Borel.
\item[($\beta$)] For each $f,g\in \cD$, the map $X \ni x\mapsto \langle f(x),g(x)\rangle_{\cH_{x}}$ is 
Borel.
\item[($\gamma$)]  The set of functions $\{ \tilde f\mid f\in\cD\}\cup \{\pi\}$ separates the points of the 
bundle $X\ast\bH$.
\end{itemize}

Since pointwise limits of sequences preserve the Borel property of maps, it follows that, when viewing $\cD$ as a subset of 
$\cF_{\widetilde \bH}(X)$, it has the following properties, with respect to the $\sigma$-algebras $\widetilde \Sigma$ and
$\widetilde \Theta$.
\begin{itemize}
\item[($\alpha$)] For each $f\in \cD$, the map $\tilde f\colon X\ast\widetilde\bH\ra \CC$, 
defined by $\tilde f(x,h):=\langle f(x),h\rangle_{\widetilde\cH_{x}}$, is Borel.
\item[($\beta$)] For each $f,g\in \cD$, the map $X \ni x\mapsto \langle f(x),g(x)\rangle_{\widetilde\cH_{x}}$ is 
Borel.
\item[($\gamma$)]  The set of functions $\{ \tilde f\mid f\in\cD\}\cup \{\pi\}$ separates the points of the 
bundle $X\ast\widetilde \bH$.
\end{itemize}
In view of Definition~\ref{d:bb}, it follows that $(X\ast\widetilde \bH,\widetilde \Theta)$ is a Borel bundle of Hilbert spaces.

(2) In view of the definitions of direct integrals of Hilbert spaces, from assumptions we easily get that, for each 
$\lambda\in\Lambda$, the Hilbert space $L^2_\lambda(X\ast\bH,\mu)$ is isometrically included in the Hilbert space
$L^2(X\ast\widetilde\bH,\widetilde \mu)$, in particular the locally Hilbert space $L^2_{\mathrm{loc}}(X\ast\bH,\mu)$ is
isometrically included in $L^2(X\ast\widetilde\bH,\widetilde \mu)$. 

In order to prove that $L^2_{\mathrm{loc}}(X\ast\bH,\mu)$ is dense in $L^2(X\ast\widetilde\bH,\widetilde \mu)$, let $f\in 
L^2_{\mathrm{loc}}(X\ast\bH,\mu)$ be such that it is orthogonal to $L^2(X\ast\widetilde\bH,\widetilde \mu)$, equivalently,
$f$ is orthogonal to $L^2_\lambda(X\ast\bH,\mu)$, for all $\lambda\in\Lambda$. 
This implies that, for each $\lambda\in\Lambda$, the function
$f|_{X_\lambda}$ is null $\mu_\lambda$-a.e.\ on $X_\lambda$ and then, since $\Lambda$ is countable, it follows that $f$ is null 
$\widetilde\mu$-a.e.\ on $X$.
 \end{proof}
 
 One of the most important consequences of the previous theorem is that, for each $\lambda\in\Lambda$, 
 we can view the Hilbert space $L^2_\lambda(X\ast\bH,\mu)$ as a direct integral of Hilbert spaces.
 
 \begin{corollary}\label{c:dihal} 
 With notation and assumptions as in Theorem~\ref{t:dintcomp}, for each $\lambda \in\Lambda$, let 
 $\widetilde\mu_\lambda$ denote the measure on $(X,\widetilde \Sigma)$ defined by
 \begin{equation}
 \widetilde\mu_\lambda(A):=\mu_\lambda(X_\lambda\cap A),\quad A\in\widetilde\Sigma.
 \end{equation}
 Then the Hilbert spaces $L^2_\lambda(X\ast\bH,\mu)$ and $L^2(X\ast\widetilde\bH,\widetilde\mu_\lambda)$ coincide, 
 in particular $L^2_\lambda(X\ast\bH,\mu)$ is a direct integral of Hilbert spaces.
 \end{corollary}
 
\subsection{Locally Decomposable and Locally Diagonalisable Operators} \label{ss:ldldo}
We continue with notation and assumptions as in Definition~\ref{d:bblhs}. So, we have a direct integral of locally 
Hilbert spaces of the analytic Borel bundle $(X\ast\bH,\pi)$, for which we use the simpler notation 
$L^2_{\mathrm{loc}}(X\ast \bH,\mu)$.  In addition, we assume throughout this section that the directed set $\Lambda$ is 
countable. Since all the main results in the following subsections are proven under this assumption, this is not a restriction. In 
view of Corollary \ref{c:dihal}, the Hilbert spaces $L^2_\lambda(X\ast\bH,\mu)$ are direct integral Hilbert spaces with respect to 
the Borel bundle of Hilbert spaces $(X\ast\widetilde\bH,\widetilde\Theta)$ and the measure $\widetilde\mu_\lambda$.

\begin{definition} \label{d:locdec}
With notations and assumptions as before, let $T \in \mathcal{B}_{\text{loc}} \bigl( L^2_{\text{loc}}(X \ast \mathbf{H}, \mu) \bigr)$, 
with $T= \varprojlim\limits_{\lambda \in \Lambda} T_\lambda$, where $T_\lambda\in\cB(L^2_\lambda(X,\mu))$, for each 
$\lambda\in\Lambda$.

\nr{1} $T$ is called \emph{locally decomposable} if, for each $\lambda\in\Lambda$, the operator 
$T_\lambda$ is decomposable, that is,
there exists a family $\{T_{\lambda,x}\}_{x\in X_\lambda}$, where 
$T_{\lambda,x}\in \cB(\cH_{x,\lambda})$, for all $x\in X_\lambda$, such that, for all $f\in L^2_\lambda (X\ast\bH,\mu)$
 we have 
\begin{equation}\label{e:talafex}
(T_\lambda f)(x)=T_{\lambda,x} f(x),\quad \mu\mbox{-a.e. } x\in X_\lambda.
\end{equation}

\nr{2} $T$ is called \emph{locally diagonalisable} if, for each $\lambda\in\Lambda$, the operator 
$T_\lambda$ is diagonalisable, that is, there exists $\Phi_\lambda \in L^\infty(X,\mu)$ with
$\esssupp(\Phi_\lambda)\subseteq X_\lambda$, such 
that $T_\lambda$ is the multiplication operator with $\Phi_\lambda$, $\mu$-a.e.\ on $X_\lambda$, that is,
for all $f\in L^2_\lambda(X\ast\bH,\mu)$, we have
\begin{equation}\label{e:talafexa}
(T_\lambda f)(x)=\Phi_\lambda(x) f(x),\quad  \mu\mbox{-a.e. } x\in X.
\end{equation}

We denote by $\mathcal{B}^{\mathrm{dec}}_{\mathrm{loc}} \big ( L^2_{\mathrm{loc}}(X \ast \mathbf{H}, \mu) \big )$
the collection of all locally decomposable operators and we denote by 
$\mathcal{B}^{\mathrm{diag}}_{\mathrm{loc}} \big ( L^2_{\mathrm{loc}}(X \ast \mathbf{H}, \mu) \big )$
the collection of all locally diagonalisable operators.
\end{definition}

\begin{theorem}\label{t:locdecdeg}
With notation and assumptions as in Definition~\ref{d:bblhs}, consider the direct integral of locally 
Hilbert spaces $L^2_{\mathrm{loc}}(X\ast \bH,\mu)$ with respect to 
the analytic Borel bundle $(X\ast\bH,\pi)$ and, in addition, we assume that the directed set $\Lambda$ is 
countable.

\nr{1} $\mathcal{B}^{\mathrm{dec}}_{\mathrm{loc}} \big ( L^2_{\mathrm{loc}}(X \ast \mathbf{H}, \mu) \big )$  and
$\mathcal{B}^{\mathrm{diag}}_{\mathrm{loc}} \big ( L^2_{\mathrm{loc}}(X \ast \mathbf{H}, \mu) \big )$  are locally von 
Neumann algebras and the latter is Abelian..

\nr{2} For each $\epsilon\in\Lambda$ and for each $S\in\cB_{\mathrm{dec}}(L^2_\epsilon(X,\mu))$, 
there exists
$T=\varprojlim\limits_{\lambda \in \Lambda} T_\lambda\in \mathcal{B}^{\mathrm{dec}}_{\mathrm{loc}} \big ( L^2_{\mathrm{loc}}(X 
\ast \mathbf{H}, \mu) \big )$ such that $T_\epsilon=S$.

\nr{3} For each $\epsilon\in\Lambda$ and for each $S\in\cB_{\mathrm{diag}}(L^2_\epsilon(X,\mu))$, 
there exists
$T=\varprojlim\limits_{\lambda \in \Lambda} T_\lambda\in \mathcal{B}^{\mathrm{diag}}_{\mathrm{loc}} \big( 
L^2_{\mathrm{loc}}(X \ast \mathbf{H}, \mu) \big )$ such that $T_\epsilon=S$.

\nr{4} $\mathcal{B}^{\mathrm{dec}}_{\mathrm{loc}} \big ( L^2_{\mathrm{loc}}(X \ast \mathbf{H}, \mu) \big )=
\mathcal{B}^{\mathrm{diag}}_{\mathrm{loc}} \big ( L^2_{\mathrm{loc}}(X \ast \mathbf{H}, \mu) \big )^\prime$. 

\nr{5}
$\mathcal{B}^{\mathrm{diag}}_{\mathrm{loc}} \big ( L^2_{\mathrm{loc}}(X \ast \mathbf{H}, \mu) \big )=
\mathcal{B}^{\mathrm{dec}}_{\mathrm{loc}} \big ( L^2_{\mathrm{loc}}(X \ast \mathbf{H}, \mu) \big )^\prime$.

\end{theorem}

\begin{proof} (1) With respect to Definition~\ref{d:lvn},  in order to prove that
$\mathcal{B}^{\mathrm{dec}}_{\mathrm{loc}} \big ( L^2_{\mathrm{loc}}(X \ast \mathbf{H}, \mu) \big )$ is a locally von Neumann 
algebra, the only axiom that needs to be verified is (LVN3). So, 
$T= \varprojlim\limits_{\lambda \in \Lambda} T_\lambda\in 
\mathcal{B}^{\mathrm{dec}}_{\mathrm{loc}} \big ( L^2_{\mathrm{loc}}(X \ast \mathbf{H}, \mu) \big )$ and fix 
$\epsilon\in\Lambda$ arbitrary. We have to prove that $TP_\epsilon$ is decomposable as well. Since $T$ is decomposable,
for each $\lambda\in\Lambda$, there exists $\{T_{\lambda,x}\}_{x\in X_\lambda}$ such that, for each 
$f\in L_\lambda^2(X\ast\bH,\mu)$, we have 
\begin{equation}\label{e:telaf}
(T_\lambda f)(x)=T_{\lambda,x}f(x),\quad \mu\mbox{-a.e.}\ x\in X_\lambda.\end{equation}
 Also,
\begin{equation*}
(TP_\epsilon)_\lambda=P_\lambda TP_\epsilon \colon \cH_\lambda\ra \cH_\lambda,\quad \lambda\in \Lambda,
\end{equation*}
hence, for each $f\in L^2_\lambda(X\ast\bH,\mu)$, each $\lambda\in\Lambda$, and $\mu$-a.e.\ $x\in X$, in view of 
Proposition~\ref{p:dintrep}, we have
\begin{align*}
\bigl((TP_\epsilon)_\lambda f\bigr)(x) & = (P_\lambda TP_\epsilon f)(x) = (P_\lambda T)(P_\epsilon f)(x) \\
& = (P_\lambda T)\chi_{X_\epsilon}(x)f(x) = \chi_{X_\epsilon} (x) (P_\lambda T f)(x) = \chi_{X_\epsilon} (x) ( T_\lambda f)(x) \\
\intertext{which, in view of \eqref{e:telaf}, equals}
& = \chi_{X_\epsilon} (x)  T_{\lambda,x} f(x),\quad \mu\mbox{-a.e.}\ x\in X.
\end{align*}
In view of Definition~\ref{d:locdec}, this means that $TP_\epsilon$ is locally 
decomposable, and hence the axiom (LVN3) is fulfilled.

As before, in order to prove that $\mathcal{B}^{\mathrm{diag}}_{\mathrm{loc}} 
\big ( L^2_{\mathrm{loc}}(X \ast \mathbf{H}, \mu) \big )$ is a locally von Neumann algebra,
 the only axiom that should be verified is (LVN3), which is a direct consequence of the statement just proven before and 
Definition~\ref{d:locdec}. Indeed, with respect to the notation from item (1), if 
 $T= \varprojlim\limits_{\lambda \in \Lambda} T_\lambda\in 
\mathcal{B}^{\mathrm{diag}}_{\mathrm{loc}} \big ( L^2_{\mathrm{loc}}(X \ast \mathbf{H}, \mu) \big )$ then, for each 
$\lambda\in\Lambda$, there exists $\Phi_\lambda \in L^\infty(X,\mu)$ with $\esssupp(\Phi_\lambda)\subseteq X_\lambda$
and such that \eqref{e:talafexa} holds. For arbitrary but fixed 
$\epsilon\in \Lambda$, from the calculation performed before, it follows
that $T_{\lambda}P_\epsilon$ is the multiplication operator with $\chi_{X_\epsilon}\Phi_\lambda \in L^\infty(X,\mu)$
with $\esssupp(\chi_{X_\lambda}\Phi_\lambda)\subseteq X_\lambda$, hence $TP_\epsilon$ is a locally diagonalisable operator.

(2) If $S\in\cB_{\mathrm{dec}}(L^2_\epsilon(X,\mu))$, for some fixed $\epsilon\in\Lambda$, then there exists 
$\{S_x\}_{x\in X_\epsilon}$ such that, for any $f\in L^2_\epsilon(X\ast\bH,\mu)$, we have $(Sf)(x)=S_x f(x)$, 
$\mu$-a.e.\ $x\in X_\epsilon$. For arbitrary $\lambda\in\Lambda$ different than $\epsilon$, 
let $T_x=S_x$ if $x\in X_\epsilon\cap X_\lambda$,
and $T_x=0$ if $x\in X\setminus (X_\epsilon\cap X_\lambda)$. 
Then define $T_\lambda\colon L^2_\lambda(X\ast\bH,\mu)\ra L^2_\lambda(X\ast\bH,\mu)$ by $(T_\lambda f)(x)
:=T_{x} f(x)$, for all $x\in X$. Since $S$ is bounded, it follows that $T_\lambda$ is bounded in 
$L^2_\lambda(X\ast\bH,\mu)$. With this definition and in view of Proposition~\ref{p:lbo2}, it is easy to see the net $
(T_\lambda)_{\lambda \in \Lambda}$ uniquely defines a locally bounded operator 
$T=\varprojlim\limits_{\lambda \in \Lambda} T_\lambda$ which, in view of Definition~\ref{d:locdec}, is locally decomposable
and $T_\epsilon=S$.

(3) This statement, concerning diagonalisable operators, is a special case of the statement concerning decomposable operators 
just proven, in which $S$ is an operator of multiplication with some $\Phi\in L^\infty(X,\mu)$ with 
$\esssupp(\Phi)\subseteq X_\epsilon$.

(4) If $T= \varprojlim\limits_{\lambda \in \Lambda} T_\lambda\in 
\mathcal{B}^{\mathrm{dec}}_{\mathrm{loc}} \big ( L^2_{\mathrm{loc}}(X \ast \mathbf{H}, \mu) \big )$ then,
for each $\lambda\in\Lambda$, there exists $\{T_{\lambda,x}\}_{x\in X_\lambda}$ such that, for each 
$f\in L_\lambda^2(X\ast\bH,\mu)$, the identity \eqref{e:telaf} holds.
Similarly, if 
$S= \varprojlim\limits_{\lambda \in \Lambda} S_\lambda\in 
\mathcal{B}^{\mathrm{diag}}_{\mathrm{loc}} \big ( L^2_{\mathrm{loc}}(X \ast \mathbf{H}, \mu) \big )$ then,
for each $\lambda\in\Lambda$, there exists $\Phi_\lambda\in L^\infty(X,\mu)$ with 
$\esssupp(\Phi_\lambda)\subseteq X_\lambda$, such that, for each $f\in L_\lambda^2(X\ast\bH,\mu)$, we have
\begin{equation}\label{e:telafux}
(S_\lambda f)(x)=\Phi_{\lambda}(x)f(x),\quad \mu\mbox{-a.e.}\ x\in X.\end{equation}
Then, for each $\lambda\in\Lambda$ and each $f\in L_\lambda^2(X\ast\bH,\mu)$, we have
\begin{equation}\label{e:telas} 
(T_\lambda S_\lambda f)(x)=T_{\lambda,x}\Phi_\lambda(x) f(x)=\Phi_\lambda(x) T_{\lambda,x} f(x)=(S_\lambda T_\lambda f)(x),
\quad \mu\mbox{-a.e. } x\in X.
\end{equation}
This shows that $TS=ST$, see \eqref{e:setep}.

From \eqref{e:telas} we get
\begin{equation*}
\mathcal{B}^{\mathrm{dec}}_{\mathrm{loc}} \big ( L^2_{\mathrm{loc}}(X \ast \mathbf{H}, \mu) \big )\subseteq
\mathcal{B}^{\mathrm{diag}}_{\mathrm{loc}} \big ( L^2_{\mathrm{loc}}(X \ast \mathbf{H}, \mu) \big )^\prime .
\end{equation*}
In order to prove the converse inclusion, let 
$T \in \mathcal{B}^{\mathrm{diag}}_{\mathrm{loc}} \big ( L^2_{\mathrm{loc}}(X \ast \mathbf{H}, \mu) \big )^\prime$ be arbitrary.
By means of Theorem~\ref{t:decdeg}, 
for each $\lambda\in\Lambda$, the operator $T_\lambda$ is decomposable, hence $T$ is locally decomposable.

(5) From \eqref{e:telas} we get 
\begin{equation*}
\mathcal{B}^{\mathrm{diag}}_{\mathrm{loc}} \big ( L^2_{\mathrm{loc}}(X \ast \mathbf{H}, \mu) \big )\subseteq
\mathcal{B}^{\mathrm{dec}}_{\mathrm{loc}} \big ( L^2_{\mathrm{loc}}(X \ast \mathbf{H}, \mu) \big )^\prime.
\end{equation*}
In order to prove the converse inclusion, let 
$T \in \mathcal{B}^{\mathrm{dec}}_{\mathrm{loc}} \big ( L^2_{\mathrm{loc}}(X \ast \mathbf{H}, \mu) \big )^\prime$ be arbitrary.
By means of Theorem~\ref{t:decdeg}, for each $\lambda\in\Lambda$, the operator $T_\lambda$ is diagonalisable, hence $T$ is locally diagonalisable.
\end{proof}

\subsection{Functional Model and Reduction Theory of Abelian Locally von Neumann Algebras.}\label{s:fm}
First, we obtain a functional representation of Abelian locally von Neumann algebras in a form which is 
a local variant of Theorem~\ref{t:abvn} and hence it can be viewed as a reduction theory for these operator 
algebras. In particular, 
it is more general than the direct integral of locally Hilbert spaces representation since it requires no separability condition.
A directed set $\Lambda$ is called \emph{sequentially finite} if it has the following properties.
\begin{itemize}
\item[(c1)] There exists a sequence $(\epsilon_m)_{m\geq 1}$ of elements of $\Lambda$ such that for 
each $m\geq 1$ we have $\epsilon_m\leq \epsilon_{m+1}$.
\item[(c2)] For each $\lambda\in \Lambda$ there exists $m\geq 1$ such that $\lambda\leq\epsilon_m$.
\item[(c3)]  For each $m\geq 1$ the set of all $\lambda\in\Lambda$ such that $\lambda\leq\epsilon_m$ 
is finite.
\end{itemize}
Note that any sequentially finite directed set is countable.

\begin{theorem}\label{t:avn1} Assume that the directed set $\Lambda$ is sequentially finite,
consider a representing locally Hilbert space
$\cH=\varinjlim_{\lambda\in\Lambda}\cH_\lambda$, and let $\cM=\varprojlim_{\lambda\in\Lambda} \cM_\lambda$ 
be an Abelian locally von Neumann algebra in $\Lloc(\cH)$. Then, the following assertions hold true.
\begin{itemize} 
\item[(a)] For each $n\in\NN\cup\{\infty\}$,
there exists a strictly inductive system of locally finite 
measure spaces $((X_{\lambda,n},\Omega_{\lambda,n},\mu_{\lambda,n}))_{\lambda\in\Lambda}$ and let 
$(X_n,\Omega_n,\mu_n)$ denote its inductive limit.
\item[(b)] For each $\lambda\in\Lambda$ there exists a decomposition
\begin{equation}\label{e:hala}
\cH_\lambda =\cH_{\lambda,\infty}\oplus\bigoplus_{n=1}^\infty \cH_{\lambda,n},
\end{equation}
such that, for each $n\in\NN\cup\{\infty\}$ we have a strictly inductive net of Hilbert spaces 
$(\cH_{\lambda,n})_{\lambda\in\Lambda}$.
\item[(c)] For each $\lambda\in\Lambda$ and for each $n\in\NN\cup\{\infty\}$, the Hilbert space 
$\cH_{\lambda,n}$ is invariant under the von Neumann algebra $\cM_\lambda$ and, 
letting $P_{\cH_{\lambda,n}}$ 
denote the orthogonal projection of $\cH_\lambda$ onto $\cH_{\lambda,n}$, the reduced von Neumann algebra
$\cM_{\lambda,\cH_{\lambda,n}}:=P_{\cH_{\lambda,n}} \cM_\lambda|_{\cH_{\lambda,n}}$ is spatially isomorphic
to the $n$-fold amplification $I_n\otimes L^\infty(X_{\lambda,n},\mu_{\lambda,n})$, more precisely, there exists a 
unitary operator $U_{\lambda,n}\colon \cH_{\lambda,n}\ra
\ell^2_n\otimes L^2_\lambda(X_{n},\mu_{n})$ 
such that the map
\begin{equation*}
\cM_{\lambda,\cH_{\lambda,n}}\ni Z\mapsto U_{\lambda,n}ZU_{\lambda,n}^*\in I_n\otimes 
L^\infty(X_{\lambda,n},\mu_{\lambda,n})
\end{equation*}
is an isomorphism of von Neumann algebras.
\item[(d)] The decomposition \eqref{e:hala} is unique with the property (c).
\item[(e)] Letting
\begin{equation}\label{e:ulam}
U_\lambda:=U_{\lambda,\infty}\oplus\bigoplus_{n\geq 1}U_{\lambda,n}\colon
\cH_\lambda \ra \bigl(\ell^2_\infty\otimes L^2_\lambda(X_{\infty},\mu_\infty)\bigr)\oplus
\bigoplus_{n\geq 1}\bigl(\ell^2_n\otimes L^2_\lambda(X_{n},\mu_n)\bigr),\quad \lambda\in\Lambda,
\end{equation}
we get a net of unitary operators $(U_{\lambda})_{\lambda\in\Lambda}$ which yields the locally unitary 
operator
\begin{equation}\label{e:uvaram}
U=\varprojlim_{\lambda\in\Lambda}U_\lambda\colon \cH=\varinjlim_{\lambda\in\Lambda}\cH_\lambda\ra 
\varinjlim_{\lambda\in\Lambda}\biggl(\bigl(\ell^2_\infty\otimes L^2_\lambda(X_{\infty},
\mu_{\infty})\bigr)
\oplus \bigoplus_{n\geq 1} \bigl(\ell^2_n\otimes L^2_\lambda(X_{n},\mu_{n})\bigr)\biggr),
\end{equation}
such that
\begin{equation}\label{e:uulah}
U^*\cM U=\varprojlim_{\lambda\in\Lambda}\bigl((I_\infty\otimes L^\infty(X_{\lambda,\infty},\mu_{\lambda,\infty}))\oplus
\bigoplus_{n\geq 1} (I_n\otimes L^\infty(X_{\lambda,n},\mu_{\lambda,n}))\bigr).
\end{equation}
\end{itemize}
\end{theorem}

\begin{proof} The proof follows to a large extent the proof of Theorem~4.16 in \cite{Gheondea4} but there are 
some details that go on slightly different paths hence we provide a full proof for the convenience of the readers.
Throughout this proof, if $\Lambda$ is a directed set and $\epsilon\in\Lambda$, we use the notation
\begin{equation}\label{e:lal}
\Lambda_\epsilon:=\{\lambda\in\Lambda \mid \lambda\leq \epsilon\}.
\end{equation}

We construct inductively, following $m\in \NN$, where $(\epsilon_m)_{m\geq 1}$ 
is the sequence of elements in $\Lambda$ subject to the conditions (c1) through (c3),
the strictly inductive systems of locally finite measure spaces 
$((X_{\lambda,n},\Omega_{\lambda,n},\mu_{\lambda,n}))_{\lambda\in\Lambda}$, for all $n\in\NN\cup\{\infty\}$, 
that satisfy the following. \medskip

\textbf{Statement S.} \emph{For each $m\in \NN$ and for each $\lambda\in \Lambda_{\epsilon_m}$,
there exists a unique decomposition 
\begin{equation}\label{e:hepsi}
\cH_\lambda=\cH_{\lambda,\infty}\oplus\bigoplus_{n\geq 1}\cH_{\lambda,n}
\end{equation}  
such that, for each $n\in\NN\cup\{\infty\}$,
\begin{itemize}
\item[(i)] $\cH_{\lambda,n}$ is invariant under  the von Neumann algebra $\cM_\lambda$ and, consequently, 
the compression of $\cM_\lambda$ to $\cH_{\lambda,n}$,
\begin{equation*}
\cM_{\lambda,\cH{\lambda,_n}}:=\{P_{\cH_{\lambda,n}}T|_{\cH_{\lambda,n}}\mid T\in\cM_\lambda\}\subseteq \cB(\cH_{\lambda,n}),
\end{equation*}
is a commutative von Neumann algebra;
\item[(ii)] there exists a unitary operator $U_{\lambda,n}\colon \cH_{\lambda,n}\ra \ell^2_n\otimes 
L^2_\lambda(X_{n},\mu_{n})$ that implements the spatial isomorphism of von Neumann algebras
\begin{equation}\label{e:phie}
I_n\otimes L^\infty(X_{\lambda,n},\mu_{\lambda,n})\ni(I_n\otimes \phi)\mapsto
U_{\lambda,n}^* (I_n\otimes M_\phi)U_{\lambda,n}\in \cM_{\lambda,\cH_n}.
\end{equation}
\end{itemize}
}

We first make the observation that, without any loss of generality, 
we can 
assume that for each $m\in\NN$ we have $\epsilon_m\neq \epsilon_{m+1}$ and that $\cH_{\epsilon_m}$ is a 
proper subspace of $\cH_{\epsilon_{m+1}}$. 

For $m=1$, with notation as in \eqref{e:lal}, we consider the finite set 
$\Lambda_{\epsilon_1}=\{\lambda_{0,1},\ldots,\lambda_{0,k_0}\}$ and 
note that the commutative von Neumann algebra $\cM_{\epsilon_1}$ 
commutes with all orthogonal projections $P_{\lambda_{0,1},\epsilon_1},\ldots, P_{\lambda_{0,k_0},\epsilon_1}$. 
Recall that all these operators act in the Hilbert space $\cH_{\epsilon_1}$, more precisely, 
$P_{\lambda_{0,1},\epsilon_1}$ is 
the orthogonal projection of $\cH_{\epsilon_1}$ onto $\cH_{\lambda_{0,1}}$ and so on.
Because the projections 
$P_{\lambda_{0,1},\epsilon_1},\ldots, P_{\lambda_{0,k_0},\epsilon_1}$ mutually commute, using the ranges 
of all nonzero projections which are finite products of 
projections of type $P_{\lambda_{0,i},\epsilon_1}$ and $I_{\cH_{\epsilon_1}}-P_{\lambda_{0,j},\epsilon_1}$, for
$i,j\in\{1,\ldots,k_0\}$,
the Hilbert space $\cH_{\epsilon_1}$ 
can be represented as an orthogonal sum of subspaces $\cH_{\epsilon_1,1},\ldots,\cH_{\epsilon_1,l_1}$
\begin{equation}
\cH_{\epsilon_1}=\cH_{\epsilon_1,1}\oplus\cdots\oplus\cH_{\epsilon_1,l_1},
\end{equation} 
all of them
invariant under $\cM_{\epsilon_1}$ and such that, for each $j\in \{1,\ldots,k_0\}$, the Hilbert space 
$\cH_{\lambda_{0,j}}$ is an orthogonal sum of some of these subspaces, 
$\cH_{\epsilon_1,p_1},\ldots,\cH_{\epsilon_1,p_{s_j}}$, that is
\begin{equation}\label{e:haloj}
\cH_{\lambda_{0,j}}= \cH_{\epsilon_1,p_1}\oplus \cdots \oplus \cH_{\epsilon_1,p_{s_j}}.
\end{equation}

For each $l=1,\ldots,l_1$, we apply Theorem~\ref{t:abvn} to the commutative reduced von Neumann algebra 
${\cM_{\epsilon_1}}_{\cH_{\epsilon_1,l}}$ and get a unique decomposition
\begin{equation}
\cH_{\epsilon_1,l}= \cH_{\epsilon_1,l,\infty}\oplus \bigoplus_{n\geq 1}\cH_{\epsilon_1,l,n},
\end{equation} 
such that, for each $n\in \NN\cup\{\infty\}$,
\begin{itemize}
\item[(i)] $\cH_{\epsilon_1,l,n}$ is invariant under $\cM_{\epsilon_1}$ and, consequently, 
$\cM_{\epsilon_1,\cH_{\epsilon_1,l,n}}$, the compression of $\cM_{\epsilon_1}$ to $\cH_{\epsilon_1,l,n}$, 
is a commutative von Neumann algebra in $\cB(\cH_{\epsilon_1,l,n})$;
\item[(ii)] there exists a locally finite measure space 
$(X_{\epsilon_1,l,n},\Omega_{\epsilon_1,l,n},\mu_{\epsilon_1,l,n})$  and a unitary operator $U_{\epsilon_1,l,n}
\colon \cH_{\epsilon_1,l,n}\ra \ell^2_n\otimes L^2(X_{\epsilon_1,l,n},\mu_{\epsilon_1,l,n})$ such that, the mapping
\begin{equation*}
I_n\otimes L^\infty(X_{\epsilon_1,l,n},\mu_{\epsilon_1,l,n})\ni I_n\otimes \phi \mapsto U_{\epsilon_1,l,n}^* 
(I_n\otimes M_\phi) U_{\epsilon_1,l,n}\in \cM_{\epsilon_1,\cH_{\epsilon_1,l,n}}
\end{equation*}
is an isomorphism of von Neumann algebras.
\end{itemize}

For each $j\in \{1,\ldots,k_0\}$, correspondingly to the decomposition \eqref{e:haloj}, and for each 
$n\in\NN\cup\{\infty\}$, we define
\begin{equation}\label{e:xala}
X_{\lambda_{0,j},n}:= \bigsqcup_{s=1}^{s_j} X_{\epsilon_1,s,n},
\end{equation}
$\Omega_{\lambda_{0,j},n}$ is the $\sigma$-algebra generated by $\Omega_{\epsilon_n,s,n}$ for all
$s=1,\ldots,s_j$, and the measure $\mu_{\lambda_{0,j},n}$ is defined by 
\begin{equation*}\mu_{\lambda_{0,j},n}(A):=\sum_{s=1}^{s_j} \mu_{\epsilon_1,s,n}(A\cap X_{\epsilon_1,s,n}),
\quad A\in \Omega_{\lambda_{0,j},n}.\end{equation*}
Then, for each $n\in\NN\cup\{\infty\}$, the locally finite measure spaces 
$(X_{\lambda_{0,j},n},\Omega_{\lambda_{0,j},n},\mu_{\lambda_{0,j},n})$, for $j\in \{1,\ldots,k_0\}$, 
make a strictly inductive system. In view of \eqref{e:xala}, we define the unitary operators $U_{\lambda_{0,j},n}$
as the direct sum of the unitary operators $U_{\epsilon_1,s,n}$ for $s=1,\ldots,s_j$, for each $j=1,\ldots,k_0$. This 
concludes the proof of the first step of induction.

For the general step of induction, let us assume that, for a fixed $m\geq 1$ and for all $n\in\NN\cup\{\infty\}$, 
we have a strictly inductive system of measure spaces 
$(X_{\lambda,n},\Omega_{\lambda,n},\mu_{\lambda,n})_{\lambda \in\Lambda_{\epsilon_m}}$ as in Statement S. 
With notation as in \eqref{e:lal}, by 
assumption, $\Lambda_{\epsilon_{m+1}}$ is finite and contains $\Lambda_{\epsilon_m}$. We split
\begin{equation*}
\Lambda_{\epsilon_{m+1}}=\Lambda_{\epsilon_m}\sqcup (\Lambda_{\epsilon_{m+1}}\setminus \Lambda_{\epsilon_m}),
\end{equation*}
hence
\begin{equation*}
\Lambda_{\epsilon_m}=\{\lambda_{m,1},\ldots,\lambda_{m,k_m}\},\quad \Lambda_{\epsilon_{m+1}}\setminus 
\Lambda_{\epsilon_m}=\{\lambda_{m,k_m+1},\ldots,\lambda_{m,k_m+k_{m+1}}\},
\end{equation*}
for some natural numbers $k_m$ and $k_{m+1}$.
Then, the von Neumann algebra $\cM_{\epsilon_{m+1}}$ commutes with all  projections 
$P_{\lambda_{m,1},\epsilon_{m+1}}$ through $P_{\lambda_{m,k_m+k_{m+1}},\epsilon_{m+1}}$, more precisely
$P_{\lambda_{m,1},\epsilon_{m+1}}$
is the projection of $\cH_{\epsilon_{m+1}}$ onto $\cH_{\lambda_{m,1}}$, and so on. Recall that all these 
operators act in $\cH_{\epsilon_{m+1}}$ and mutually commute. Then, by means of the 
ranges of all nonzero projections which are finite products of 
projections $P_{\lambda_{m,i},\epsilon_{m+1}}$ and 
$I_{\cH_{\epsilon_{m+1}}}-P_{\lambda_{m,j},\epsilon_{m+1}}$, for all $i,j\in\{1,\ldots,k_m\}$,
the Hilbert space $\cH_{\epsilon_{m+1}}$ can be split into 
mutually orthogonal subspaces $\cH_{\epsilon_{m+1},1},\ldots,\cH_{\epsilon_{m+1},l_m},
\cH_{\epsilon_{m+1},l_m+1},\ldots,\cH_{\epsilon_{m+1},l_m+l_{m+1}}$, all of them invariant under 
$\cM_{\epsilon_{m+1}}$,
\begin{equation}\label{e:hememu}
\cH_{\epsilon_{m+1}}=\bigoplus_{j=1}^{l_m} \cH_{\epsilon_{m+1},j}\oplus \bigoplus_{j=1}^{l_{m+1}} 
\cH_{\epsilon_{m+1},l_m+j},
\end{equation}
such that,
\begin{equation}\label{e:hemem}
\cH_{\epsilon_m} = \bigoplus_{j=1}^{l_m} \cH_{\epsilon_{m+1},j}.
\end{equation}
For each $k=1,\ldots,k_m$, there exist indices $j_1,\ldots,j_{p_k}\in\{1,\ldots,l_m\}$, $p_k\leq l_m$, such that
\begin{equation}\label{e:celam}
\cH_{\lambda_{m,k}}=\bigoplus_{i=1}^{p_k} \cH_{\epsilon_{m+1},j_{i}},
\end{equation}
and, for each $k=k_m+1,\ldots,k_m+k_{m+1}$, there exist indices $j_1,\ldots,j_{p_k}\in \{1,\ldots,l_m+l_{m+1}\}$, 
$p_k\leq l_m+l_{m+1}$, such that \eqref{e:celam} holds.

At this step, the construction has to preserve everything was constructed at the previous steps and, because of 
that, with respect to the components $\cH_{\epsilon_{m+1},j}$, for $j=1,\ldots,l_m$, 
in the first orthogonal sum from the right hand side of
\eqref{e:hememu} and the components $\cH_{\epsilon_{m+1},l_m+j}$, for $j=1,\ldots,l_{m+1}$, in the latter 
orthogonal sum, we have to treat separately two cases: one in which we preserve the information obtained at the 
previous steps and another one in which we have to obtain new information.

\emph{First case:} $j\in \{1,\ldots,l_m\}$. 
By assumption, Statement S holds for $\lambda=\epsilon_m$ hence, there exists a unique decomposition
\begin{equation}\label{e:hepsiem}
\cH_{\epsilon_m}=\cH_{\epsilon_m,\infty}\oplus\bigoplus_{n\geq 1}\cH_{\epsilon_m,n}
\end{equation}  
and, for each $n\in\NN\cup\{\infty\}$, there exist the locally finite measure space 
$(X_{\epsilon_m,n},\Omega_{\epsilon_m,n},\mu_{\epsilon_m,n})$ and the unitary operator 
$U_{\epsilon_m,n}\colon \cH_{\epsilon_m,n}\ra \ell^2_n\otimes 
L^2(X_{\epsilon_m,n},\mu_{\epsilon_m,n})$ such that \eqref{e:phie} holds for $\lambda=\epsilon_m$.
Then, in view of \eqref{e:hemem} and the fact that all the spaces there involved are invariant under 
$\cM_{\epsilon_{m}}$, 
by considering the projection of $\cH_{\epsilon_{m+1}}$ onto $\cH_{\epsilon_{m+1},j}$,
for each $n\in\NN\cup\{\infty\}$ we can get a locally finite measure space 
$(X_{\epsilon_{m+1},j,n},\Omega_{\epsilon_{m+1},j,n},\mu_{\epsilon_{m+1},j,n})$ which is obtained by restriction 
from the locally finite measure space $(X_{\epsilon_m,n},\Omega_{\epsilon_m,n},\mu_{\epsilon_m,n})$, and a 
unitary operator $U_{\epsilon_{m+1},j,n}\colon \cH_{\epsilon_{m+1},j,n}\ra \ell^2_n\otimes 
L^2(X_{\epsilon_{m+1},j,n},\mu_{\epsilon_{m+1},j,n})$, where we have the decomposition 
\begin{equation}
\cH_{\epsilon_{m+1},j}=\cH_{\epsilon_{m+1},j,\infty}\oplus\bigoplus_{n=1}^\infty \cH_{\epsilon_{m+1},j,n},
\end{equation}
which is obtained from \eqref{e:hepsiem}, such that the unitary operator $U_{\epsilon_{m+1},j,n}$ implements 
the spatial isomorphism of the reduced von Neumann algebra $\cM_{\epsilon_{m+1},\cH_{\epsilon_{m+1},j,n}}$ 
with the $n$-times amplification von Neumann algebra 
$I_n\otimes L^\infty(X_{\epsilon_{m+1},j,n},\mu_{\epsilon_{m+1},j,n})$.

\emph{Second case:} $j\in \{l_m+1,\ldots,l_{m}+l_{m+1}\}$. In this case we have to provide new information and 
add it in an appropriate way to the information obtained in the previous steps.
For this we apply Theorem~\ref{t:abvn} and get the unique decomposition
\begin{equation}
\cH_{\epsilon_{m+1},j}=\cH_{\epsilon_{m+1},j,\infty}\oplus\bigoplus_{n=1}^\infty \cH_{\epsilon_{m+1},j,n},
\end{equation}
with all these spaces invariant under $\cM_{\epsilon_{m+1}}$, and, for each $n\in\NN\cup\{\infty\}$,
we can get a locally finite measure space 
$(X_{\epsilon_{m+1},j,n},\Omega_{\epsilon_{m+1},j,n},\mu_{\epsilon_{m+1},j,n})$
and a 
unitary operator $U_{\epsilon_{m+1},j,n}\colon \cH_{\epsilon_{m+1},j,n}\ra \ell^2_n\otimes 
L^2(X_{\epsilon_{m+1},j,n},\mu_{\epsilon_{m+1},j,n})$ which implements 
the spatial isomorphism of the reduced von Neumann algebra $\cM_{\epsilon_{m+1},\cH_{\epsilon_{m+1},j,n}}$ 
with the $n$-times amplification von Neumann algebra 
$I_n\otimes L^\infty(X_{\epsilon_{m+1},j,n},\mu_{\epsilon_{m+1},j,n})$.

Then, let $\lambda\in \Lambda_{\epsilon_{m+1}}\setminus \Lambda_{\epsilon_m}$, hence 
$\lambda=\lambda_{m,k_m+k}$ 
for some $k\in \{1,\ldots,k_{m+1}\}$ and there exist indices $j_1,\ldots,j_{p_k}\in \{1,\ldots,l_m+l_{m+1}\}$, 
$p_k\leq l_m+l_{m+1}$, such that \eqref{e:celam} holds. For each $n\in \NN\cup\{\infty\}$ 
we define  
\begin{equation}\label{e:xalam}
X_{\lambda_{m,k_m+k,n}} = \bigsqcup_{i=1}^{p_k} X_{\epsilon_{m+1},i,n}.
\end{equation}
Let $\Omega_{\lambda_{m,k_m+k,n}}$ be the $\sigma$-algebra generated by 
$\Omega_{\epsilon_{m+1},i,n}$ for $i=1,\ldots,p_k$, and define the measure 
$\mu_{\epsilon_{m+1},k_m+k,n}\colon\Omega_{\epsilon_{m+1},k_m+k,n}\ra [0,+\infty]$ by
\begin{equation}
\mu_{\epsilon_{m+1},k_m+k,n}(A):= \sum_{i=1}^{p_k} \mu_{\epsilon_{m+1},i,n}(A\cap X_{\epsilon_{m+1},p_i,n}),\quad A\in \Omega_{\epsilon_{m+1},k_m+k,n}.
\end{equation}
The unitary operators $U_{\lambda_{m,k_m+k,n}}$ are defined, in view of  \eqref{e:xalam}, as the direct sum of
operators $U_{\epsilon_{m+1},i,n}$ for $i=1,\ldots,p_k$.
Putting all these things together, it is easy 
to see that Statement S holds for $\epsilon_{m+1}$ and the general induction step is 
proven, which concludes the proof of Statement S for all $m\in\NN$.

In view of the Principle of Mathematical Induction and the technical condition (c2), Statement S proves items (a), 
(b), (c), and (d). Then, in view of Proposition~\ref{p:dirsum} and Proposition~\ref{p:tensorprod}, with the identification of the Hilbert spaces $\ell^2_n\otimes L^2(X_\lambda,n,\mu_{\lambda,n})$ and $L^2_{\ell^2_n}(X_\lambda,\mu_{\lambda,n})$ as in Remark~\ref{r:tensorprod},
we can define the representing locally Hilbert space
\begin{equation}\label{e:elal}
\varinjlim_{\lambda\in\Lambda}\biggl(\bigl(\ell^2_\infty\otimes L^2_\lambda(X_{\infty},\mu_{\infty})\bigr)\oplus
\bigoplus_{n\geq 1}\bigl(\ell^2_n\otimes L^2_\lambda(X_{n},\mu_{n})\bigr)\biggr)
\end{equation}
and, for any $\lambda\in\Lambda$ we consider the unitary operator $U_\lambda$ defined as in \eqref{e:ulam}.
By Proposition~\ref{p:lbo2}, we observe that the net $(U_\lambda)_{\lambda\in\Lambda}$ of unitary operators 
gives rise to the locally unitary operator 
\begin{equation}\label{e:uvar}
U:=\varprojlim_{\lambda\in\Lambda} U_\lambda\colon \cH=\varinjlim_{\lambda\in\Lambda} \cH_\lambda
\ra \varinjlim_{\lambda\in\Lambda} \bigl(\ell^2_\infty\otimes L^2_\lambda(X_{\infty},\mu_{\infty})\bigr)\oplus
\bigoplus_{n\geq 1}\bigl(\ell^2_n\otimes L^2_\lambda(X_{n},\mu_{n})\bigr).
\end{equation}
In view of the properties (a) through (c), we can easily see that the locally unitary operator $U$ defined at 
\eqref{e:uvar} implements the $*$-isomorphism between $\cM$ and the locally von Neumann algebra described 
at \eqref{e:elal}, as in \eqref{e:uulah}.
\end{proof}

The above theorem can be made more general by removing the assumption that the locally Hilbert space is 
representing, in which case, instead of the spatial isomorphism as in  \eqref{e:uulah}
we only obtain a $*$-isomorphism of locally von Neumann algebras that might not 
be spatial, that is, implemented by a locally unitary operator. 

\begin{corollary}
Assume that the directed set $\Lambda$ is sequentially finite, consider a locally Hilbert space
$\cH=\varinjlim_{\lambda\in\Lambda}\cH_\lambda$, and let $\cM=\varprojlim_{\lambda\in\Lambda} \cM_\lambda$ 
be an Abelian locally von Neumann algebra in $\Lloc(\cH)$. Then, 
for each $n\in\NN\cup\{\infty\}$, there exists a strictly inductive system of locally finite 
measure spaces $((X_{\lambda,n},\Omega_{\lambda,n},\mu_{\lambda,n}))_{\lambda\in\Lambda}$ such that
the locally von Neumann algebra $\cM$ is $*$-isomorphic to the locally von Neumann algebra
$\varprojlim_{\lambda\in\Lambda}\bigl((I_\infty\otimes L^\infty(X_{\lambda,\infty},\mu_{\lambda,\infty}))\oplus
\bigoplus_{n\geq 1} (I_n\otimes L^\infty(X_{\lambda,n},\mu_{\lambda,n}))\bigr)$.
\end{corollary}

Another byproduct of Theorem~\ref{t:avn1} is the representation 
of Abelian locally von Neumann algebras in terms of $L^\infty$ type algebras, without spectral multiplicity. The cost is, of course,  
more complex measure spaces.

\begin{corollary}\label{c:avn1}
Assume that the directed set $\Lambda$ is sequentially finite,
consider a representing locally Hilbert space
$\cH=\varinjlim_{\lambda\in\Lambda}\cH_\lambda$, and let $\cM=\varprojlim_{\lambda\in\Lambda} \cM_\lambda$ 
be an Abelian locally von Neumann algebra in $\Lloc(\cH)$. Then,
there exists a strictly inductive system of locally finite 
measure spaces $((X_{\lambda},\Omega_{\lambda},\mu_{\lambda}))_{\lambda\in\Lambda}$, 
and there exists a locally unitary operator
\begin{equation*}
U\colon \cH\ra \varinjlim_{\lambda\in\Lambda}L^2_\lambda(X,\mu),
\end{equation*}
such that
\begin{equation*}
U^*\cM U=\varprojlim_{\lambda\in\Lambda} L^\infty(X_{\lambda},\mu_{\lambda}).
\end{equation*}
\end{corollary}
\subsection{Direct Integral Representations of Locally von Neumann Algebras}\label{ss:dirlvna}
From Theorem~\ref{t:avn1} 
we can use the ideas and techniques as in Theorem~\ref{t:did} and, assuming in addition that all the 
Hilbert spaces $\cH_\lambda$ are separable, we can put this into a direct integral representation of an Abelian 
locally von Neumann algebra. 

Let $\big((X_{\lambda}, \Omega_{\lambda}, \mu_{\lambda})\big)_{\lambda \in \Lambda}$ be a strictly inductive system of 
measure spaces and let $(X, \Omega, \mu)$ be its inductive limit. First note that, if $E \in \Omega_{\lambda}$ then, clearly
\begin{equation*} \Omega_{\lambda, E}:= \left\{ A\cap E:\; A \in \Omega_{\lambda}\right\}
\end{equation*} 
defines a $\sigma$-algebra of subsets of $E$ and the restriction $\mu_{\lambda} \big|_E$ defines a measure on $E.$  If 
$\mathcal{G}$ is a Hilbert space and $E \in \Omega_\Lambda$, then we consider the vector space
\begin{align}\label{e:lefexet}
\cL^{2}_{\mathcal{G}, E}(X, \mu) & :=
\big\{ f \colon X \to \mathcal{G} \mid \supp(f) \subseteq E,\ f\big|_{E} \text{ is measurable w.r.t. } \Omega_{\lambda, E},\\
& \hspace{7em}\int\limits_{E} \|f(x)\|_{\mathcal{G}}^{2}\; \de\mu_{\lambda}(x) < \infty \big\},\nonumber
\end{align}
and denote by $L^{2}_{\mathcal{G}, E}(X, \mu)$ the Hilbert space obtained from $\cL^{2}_{\mathcal{G}, E}(X, \mu)$
by identification $\mu_{\lambda}$-a.e. We will use this conventional notation at item (iv) of the following theorem.

\begin{theorem}\label{t:avn2} 
Let $\big( \Lambda, \leq \big)$ be a sequentially finite directed set. Consider a representing locally Hilbert space 
$\mathcal{H} := \varinjlim\limits_{\lambda \in \Lambda} \mathcal{H}_{\lambda}$, such that $\mathcal{H}_\lambda$ is separable for 
each $\lambda \in \Lambda$, and let $\mathcal{M} = \varprojlim\limits_{\lambda \in \Lambda} \mathcal{M}_{\lambda}$ be an 
Abelian locally von Neumann algebra in $\mathcal{B}_{\mathrm{loc}}(\mathcal{H})$. Then, for every $\lambda \in \Lambda$, there 
exist a measure space $\big ( X_\lambda, \Omega_\lambda, \mu_\lambda \big)$ and a family of Hilbert spaces 
$\bG_\lambda = \{ \mathcal{G}_{x, \lambda} \}_{x \in X_\lambda}$ with the following properties.
\begin{enumerate}
\item[(i)] The net  $\big ( \big ( X_\lambda, \Omega_\lambda, \mu_\lambda \big) \big )_{\lambda \in \Lambda}$ is a strictly 
inductive system of measure spaces such that for each $\lambda \in \Lambda$, $\big ( X_\lambda, \Omega_\lambda \big ) $ is a 
standard Borel space.
\item[(ii)] For each $x \in X$, the net $\big ( \mathcal{G}_{x, \lambda}  \big )_{\lambda \in \Lambda}$ is a strictly inductive system 
of separable Hilbert spaces and let $\mathcal{G}_x := \varinjlim\limits_{\lambda \in \Lambda} \mathcal{G}_{x, \lambda}$ be the 
associated locally Hilbert space. 
\item[(iii)] $\big ( \big ( X_\lambda \ast \bG_\lambda, \pi_\lambda  \big )  \big )_{\lambda \in \Lambda}$ is a net of standard Borel 
bundles of Hilbert spaces.
\item[(iv)] If $\big ( X, \Omega, \mu \big) $ is the inductive limit of the strictly inductive system 
$\big ( \big ( X_\lambda, \Omega_\lambda, \mu_\lambda \big) \big )_{\lambda \in \Lambda}$ of measure spaces, then 
$L^2 \big ( X_{\lambda} \ast \bG_{\lambda}, \mu_{\lambda} \big ) \cong L^{2}_{\ell^2_\infty, X_{\lambda, \infty}}\big ( X, \mu \big )  
\oplus \bigoplus\limits_{n \geq 1} L^{2}_{\ell^2_n, X_{\lambda, n}}\big ( X, \mu \big )$ where, 
for each $n \in \mathbb{N} \cup \{ \infty \}$, $L^{2}_{\ell^2_n, X_{\lambda, n}}\big ( X, \mu \big )$ consists of all functions 
$f \colon X \rightarrow \ell^2_n$, such that 
\begin{equation*}
\supp(f) \subseteq X_{\lambda, n},  \;  f\big|_{X_{\lambda, n}}  \text{ is measurable w.r.t. } \Omega_{\lambda, X_{\lambda, n}},\ 
\int\limits_{X_{\lambda, n}} \|f(x)\|_{\ell^2_n}^{2}\; d\mu_{\lambda}(x) < \infty.
\end{equation*}  
\item[(v)] The net
\begin{equation*}
\big( L^{2}_{\ell^2_\infty, X_{\lambda, \infty}}\big ( X, \mu \big )  \oplus \bigoplus\limits_{n \geq 1} L^{2}_{\ell^2_n, X_{\lambda, n}}\big ( X, \mu \big )  \big)_{\lambda \in \Lambda}
\end{equation*}
forms a  strictly inductive system of Hilbert spaces. Consequently, there exists a direct integral of locally Hilbert spaces
\begin{equation*}
L^2_{\mathrm{loc}}(X \ast \bG, \mu) = \displaystyle \int_X^{\oplus,\mathrm{loc}} \mathcal{G}_x \; \mathrm{d} \mu(x) 
:= \varinjlim\limits_{\lambda \in \Lambda} L^{2}_{\ell^2_\infty, X_{\lambda, \infty}}\big ( X, \mu \big )  
\oplus \bigoplus\limits_{n \geq 1} L^{2}_{\ell^2_n, X_{\lambda, n}}\big ( X, \mu \big ),
\end{equation*}
where, $\bG := \big ( \mathcal{G}_x  \big )_{x \in X}$ is a net of locally Hilbert spaces.
\item[(vi)] There exists a locally unitary operator $U \colon \mathcal{H} \rightarrow L^2_{\mathrm{loc}}(X \ast \bG, \mu)$, such that 
\begin{equation*}
\mathcal{M} \ni Z \mapsto UZU^\ast \in  \mathcal{B}^{\mathrm{diag}}_{\mathrm{loc}} \big ( L^2_{\mathrm{loc}}(X \ast \bG, \mu)\big )
\end{equation*}
is a $\ast$-isomorphism of (Abelian) locally von Neumann algebras.
\end{enumerate}
\end{theorem}

\begin{proof} 
Under the present assumptions and in view of Theorem \ref{t:avn1}, for each 
$n \in \mathbb{N} \cup \{ \infty \}$, there exists a strictly inductive system 
$\big ( \big ( X_{\lambda, n}, \Omega_{\lambda, n}, \mu_{\lambda, n} \big) \big )_{\lambda \in \Lambda}$ of locally finite measure 
spaces. Let $( X_n, \Omega_n, \mu_n \big)$ denote its inductive limit. For each $\lambda \in \Lambda$, we get a decomposition 
of the Hilbert space $\mathcal{H}_\lambda$ as
\begin{equation} \label{e:dechal}
\mathcal{H}_\lambda = \mathcal{H}_{\lambda, \infty} \oplus \bigoplus\limits_{n \geq 1} \mathcal{H}_{\lambda, n}
\end{equation}
with the property that, for each $n \in \mathbb{N} \cup \{ \infty \}$, the net 
$\big ( \mathcal{H}_{\lambda, n}  \big )_{\lambda \in \Lambda}$ forms a strictly inductive system of Hilbert spaces. Also, for each 
$\lambda \in \Lambda$ and each 
$n \in \mathbb{N} \cup \{ \infty \}$, the Hilbert space $\mathcal{H}_{\lambda, n}$ is invariant under the 
von Neumann algebra $\mathcal{M}_\lambda$. Suppose $P_{\mathcal{H}_{\lambda, n}}$ denote the orthogonal projection of 
$\mathcal{H}_\lambda$ onto the Hilbert space $\mathcal{H}_{\lambda, n}$. Let 
$\mathcal{M}_{\lambda, \mathcal{H}_{\lambda, n}} := P_{\mathcal{H}_{\lambda, n}} 
\mathcal{M}_\lambda \big|_{{\mathcal{H}_{\lambda, n}}}$. Then, we obtain a unitary operator $U_{\lambda, n} \colon 
\mathcal{H}_{\lambda, n} \rightarrow \ell^2_n \otimes L^2_\lambda(X_n, \mu_n)$, such that
\begin{equation*}
U_{\lambda, n} \mathcal{M}_{\lambda, \mathcal{H}_{\lambda, n}} U^\ast_{\lambda, n} 
= I_n \otimes L^\infty(X_{\lambda, n}, \mu_{\lambda, n}).
\end{equation*}
The decomposition of the Hilbert space $\mathcal{H}_\lambda$ given at \eqref{e:dechal} is unique with 
the properties mentioned above. For a fixed $\lambda \in \Lambda$, using the unitary operators 
$U_{\lambda, n}$ for each $n \in \mathbb{N} \cup \{ \infty \}$, define a unitary operator
\begin{equation} \label{e:ulae}
U_\lambda := U_{\lambda, \infty} \oplus \bigoplus\limits_{n \geq 1} U_{\lambda, n} \colon 
\mathcal{H}_\lambda \rightarrow \big(\ell^2_\infty \otimes L^2_\lambda(X_\infty, \mu_\infty)\big) \oplus \bigoplus\limits_{n \geq 1} \big(\ell^2_n \otimes L^2_\lambda(X_n, \mu_n)\big).
\end{equation}
This yields a net of unitary operators $(U_\lambda)_{\lambda \in \Lambda}$ which defines a locally unitary operator 
\begin{equation}  \label{e:ulaes}
U = \varprojlim_{\lambda \in \Lambda} U_\lambda \colon \mathcal{H} = \varinjlim\limits_{\lambda \in \Lambda} 
\mathcal{H}_{\lambda} \rightarrow   \varinjlim\limits_{\lambda \in \Lambda} \big( (\ell^2_\infty \otimes L^2_\lambda(X_\infty, \mu_\infty) )
\oplus \bigoplus\limits_{n \geq 1} (\ell^2_n \otimes L^2_\lambda(X_n, \mu_n))\big),
\end{equation}
such that, 
\begin{equation}  \label{e:spiu}
U^\ast \mathcal{M} U = \varprojlim_{\lambda \in \Lambda} \Big(\big( I_\infty \otimes L^\infty(X_{\lambda, \infty}, \mu_{\lambda, \infty}) \big)\oplus \bigoplus\limits_{n \geq 1} \big(I_n \otimes L^\infty(X_{\lambda, n}, \mu_{\lambda, n}) \big )\Big).
\end{equation}

Since all Hilbert spaces $\cH_\lambda$ are separable, all locally finite measure spaces  
$(X_{\lambda,n},\Omega_{\lambda,n},\mu_{\lambda,n})$ are standard, see the last statement 
in Theorem~\ref{t:abvn}. Now, for arbitrary $\lambda\in\Lambda$, 
we construct the standard measure space $(X_\lambda,\Omega_\lambda,\mu_\lambda)$ as the disjoint union of 
all standard measure spaces $(X_{\lambda,n},\Omega_{\lambda,n},\mu_{\lambda,n})$, for 
$n\in\NN\cup\{\infty\}$. More precisely,
\begin{equation}\label{e:xelam}
X_\lambda:=X_{\lambda,\infty}\sqcup\bigsqcup_{n\in\NN} X_{\lambda,n},\quad 
\Omega_\lambda:=\{B\subseteq X\mid B\cap X_{\lambda,n}\in\Omega_{\lambda,n}\mbox{ for all }n\in\NN\cup\{\infty\}\},
\end{equation}
and the measure $\mu_\lambda$ defined by
\begin{equation}\label{e:mulam}
\mu_\lambda(B)=\mu_{\lambda,\infty}(B\cap X_{\lambda,\infty})
+\sum_{n\in\NN}\mu_{\lambda,n}(B\cap X_{\lambda,n}),\quad B\in\Omega_\lambda.
\end{equation}
Then $\big ( X_\lambda, \Omega_\lambda, \mu_\lambda \big)$ is a standard measure space.  Fix $n \in \mathbb{N}$ and, for 
each $x \in X_{\lambda, n}$, define $\mathcal{G}_{x, n} := \ell^2_n$. Then, 
$\bG_{\lambda, n} := \{ \mathcal{G}_{x, n} = \ell^2_n \}_{x \in X_{\lambda, n}}$ and we obtain a Borel bundle of Hilbert spaces 
$(X_{\lambda, n} \ast \bG_{\lambda, n}, \pi_{\lambda, n})$, see Remark \ref{r:id}. For each $\lambda \in \Lambda$ and each
$n \in \mathbb{N} \cup \{ \infty \}$,  we identify 
\begin{equation} \label{e:le2en}
L^2 \big ( X_{\lambda, n} \ast \bG_{\lambda, n}, \mu_{\lambda, n} \big ) = \int^\oplus_{X_{\lambda, n}} \mathcal{G}_{x, n} 
\ \mathrm{d} \mu_{\lambda, n} 
= \ell^2_n \otimes L^2_\lambda \big ( X_n, \mu_n \big ) \cong L^{2}_{\ell^2_n, X_{\lambda, n}}\big ( X, \mu \big ), 
\end{equation}
where, $L^{2}_{\ell^2_n, X_{\lambda, n}}\big ( X, \mu \big )$ is defined as in \eqref{e:lefexet}.

For arbitrary $\lambda \in \Lambda$ and $x \in X_\lambda$, define $\mathcal{G}_{x, \lambda} := \ell^2_n$, where 
$n \in \mathbb{N} \cup \{ \infty \}$ is the unique number  such that $x \in X_{\lambda, n}$. Consider the bundle 
$\bG_\lambda := ( \mathcal{G}_{x, \lambda} )_{x \in \mathcal{X}_\lambda}$ of Hilbert spaces and observe that for each $x \in X$ 
the net $ (\mathcal{G}_{x, \lambda} )_{\lambda \in \Lambda_\lambda}$ is a strictly inductive system of Hilbert spaces, and define 
$\mathcal{G}_x := \varinjlim\limits_{\lambda \in \Lambda_\lambda} \mathcal{G}_{x, \lambda}$ to be the associated locally Hilbert 
space. Now in view of the construction of the standard measure space $\big ( X_\lambda, \Omega_\lambda, \mu_\lambda \big)$, 
Remark \ref{r:id}, Proposition~\ref{p:di} and Definition~\ref{d:di}, we get the existence of the direct integral of Hilbert space 
$L^2 \big ( X_{\lambda} \ast \bG_{\lambda}, \mu_{\lambda} \big ) 
=  \displaystyle \int^\oplus_{X_{\lambda}} \mathcal{G}_{x, \lambda} \; \mathrm{d} \mu_{\lambda}(x)$. Suppose $\bG 
:= ( \mathcal{G}_x )_{x \in X}$. Then by following the discussion before the definition of direct integral of locally Hilbert spaces 
and by  \eqref{e:le2en}, we get
\begin{align*}
L^2_\lambda \big (X \ast  \bG, \mu \big ) \cong L^2 \big ( X_{\lambda} \ast \bG_{\lambda}, \mu_{\lambda} \big ) &=  \displaystyle \int^\oplus_{X_{\lambda}} \mathcal{G}_{x, \lambda} \; \mathrm{d} \mu_{\lambda}(x) \\
&= L^2 \big ( X_{\lambda, \infty} \ast \bG_{\lambda, \infty}, \mu_{\lambda, \infty} \big ) \oplus \bigoplus\limits_{n \geq 1} L^2 \big ( X_{\lambda, n} \ast \bG_{\lambda, n}, \mu_{\lambda, n} \big ) \\
&= \big(\ell^2_\infty \otimes L^2 \big ( X_{\lambda, \infty}, \mu_{\lambda, \infty} \big )\big) \oplus \bigoplus\limits_{n \geq 1} 
\big(\ell^2_n \otimes L^2 \big ( X_{\lambda, n}, \mu_{\lambda, n} \big )\big) \\
&\cong L^{2}_{\ell^2_\infty, X_{\lambda, \infty}}\big ( X, \mu \big )  \oplus \bigoplus\limits_{n \geq 1} L^{2}_{\ell^2_n, X_{\lambda, n}}\big ( X, \mu \big ). 
\end{align*}

We claim that the family $\big\{  L^{2}_{\ell^2_\infty, X_{\lambda, \infty}}\big ( X, \mu \big ) 
\oplus \bigoplus\limits_{n \geq 1} L^{2}_{\ell^2_n, X_{\lambda, n}}\big ( X, \mu \big )  \big \}_{\lambda \in \Lambda}$ forms a 
strictly inductive system of Hilbert spaces.  To prove this, it is enough to show that, for a fixed $n \in \mathbb{N} \cup \{ \infty \}$, 
the family $\big \{ L^{2}_{\ell^2_n, X_{\lambda, n}}\big ( X, \mu \big )  \big \}_{\lambda \in \Lambda}$ forms a strictly inductive 
system of Hilbert spaces. Let $f \in L^{2}_{\ell^2_n, X_{\lambda, n}}\big ( X, \mu \big )$. Thus, by definition, 
$f \colon X \rightarrow \ell^2_n$ is such that 
\begin{equation*}
\supp(f) \subseteq X_{\lambda, n},  \; \; f\big|_{X_{\lambda, n}} \; \text{is measurable w.r.t.}\; (X_{\lambda, n}, \Omega_{\lambda, X_{\lambda, n}}),\;\; \int_{X_{\lambda, n}} \|f(x)\|_{\ell^2_n}^{2}\; d\mu_{\lambda}(x) < \infty.
\end{equation*} 
We have  $X_{\lambda, n} \subseteq X_{\delta, n}$, whenever $\lambda \leq \delta$. This implies $\supp(f) \subseteq X_{\delta, 
n}$. Now we show that $f\big|_{X_{\delta, n}}$ is measurable with respect to  $\Omega_{\delta, X_{\delta, n}}$. To 
see this, let $F \subseteq \ell^2_n$ be any measurable subset. Then
\begin{align*}
f\big|^{-1}_{X_{\delta, n}}(F) = f\big|^{-1}_{X_{\delta, n}}(F \setminus \{0\} \cup \{ 0 \}) &= f\big|^{-1}_{X_{\delta, n}}(F \setminus \{0\}) \cup f\big|^{-1}_{X_{\delta, n}}(\{0\}) \\
&= f\big|^{-1}_{X_{\lambda, n}}(F \setminus \{0\}) \cup f\big|^{-1}_{X_{\delta, n}}(\{0\}) \\
&= f\big|^{-1}_{X_{\lambda, n}}(F \setminus \{0\}) \cup X_{\delta, n} \setminus X_{\lambda, n} \cup \{ x \in X_{\lambda, n} \mid f(x) = 0  \}
\end{align*}
Thus, $f\big|^{-1}_{X_{\delta, n}}(F)$ is a measurable set of  $X_{\delta, n}$ and hence, 
$f\big|_{X_{\delta, n}}$ is measurable with respect to  $\Omega_{\delta, X_{\delta, n}}$. Lastly, as 
$\supp(f) \subseteq X_{\lambda, n} \subseteq X_{\delta, n}$, we have
\begin{equation*}
\int_{X_{\delta, n}} \|f(x)\|_{\ell^2_n}^{2}\; d\mu_{\delta}(x) = \int\limits_{X_{\lambda, n}} \|f(x)\|_{\ell^2_n}^{2}\; d\mu_{\lambda}(x) < 
\infty,\end{equation*}
hence $f \in L^{2}_{\ell^2_n}\big ( X_{\delta, n}, \mu_{\delta, n} \big )$. This shows that 
$\big \{ L^{2}_{\ell^2_n, X_{\lambda, n}}\big ( X, \mu \big )  \big \}_{\lambda \in \Lambda}$ forms a strictly inductive system of 
Hilbert spaces. As $n \in \mathbb{N} \cup \{ \infty \}$ was arbitrarily chosen, we conclude that the family 
\begin{equation*}
\big\{  L^{2}_{\ell^2_\infty, X_{\lambda, \infty}}\big ( X, \mu \big )  \oplus \bigoplus\limits_{n \geq 1} L^{2}_{\ell^2_n, X_{\lambda, n}}\big ( X, \mu \big )  \big\}_{\lambda \in \Lambda}
\end{equation*}
forms a strictly inductive system of Hilbert spaces. Now, for $\bG = ( \mathcal{G}_x )_{x \in X}$, we get
\begin{equation*}
L^2_{\mathrm{loc}}(X \ast \bG, \mu) = \displaystyle \int_X^{\oplus_\mathrm{loc}} \mathcal{G}_x \; \mathrm{d} \mu(x) := \varinjlim\limits_{\lambda \in \Lambda} L^{2}_{\ell^2_\infty, X_{\lambda, \infty}}\big ( X, \mu \big )  \oplus \bigoplus\limits_{n \geq 1} L^{2}_{\ell^2_n, X_{\lambda, n}}\big ( X, \mu \big ).
\end{equation*}
This settles all the statement from (i) through (v).

It only remains to show the existence of the locally unitary operator 
$U \colon \mathcal{H} \rightarrow L^2_{\mathrm{loc}}(X \ast \bG, \mu)$ satisfying the condition given in statement (v). Indeed, 
the locally unitary operator 
$U \colon \mathcal{H} \rightarrow L^2_{\mathrm{loc}}(X \ast \bG, \mu)$ is constructed in a natural fashion 
by putting together $U_\lambda$ for all $\lambda \in \Lambda$ and by using the building blocks  $U_{\lambda, n}$ for all 
$n \in \mathbb{N} \cup \{ \infty \}$, see \eqref{e:ulae}. 
Finally, the locally unitary operator $U$ given at \eqref{e:spiu} satisfies
\begin{equation*}
\mathcal{M} \ni Z \mapsto UZU^\ast \in  \mathcal{B}^{\mathrm{diag}}_{\mathrm{loc}} \big ( L^2_{\mathrm{loc}}(X \ast \bG, \mu)\big ),
\end{equation*}
which can be seen from \eqref{e:spiu} and Remark \ref{r:ampdiag}.
\end{proof}
\subsection*{Acknowledgment}
 The second named author sincerely thanks the Indian Institute of Science Education and Research (IISER) Mohali, India and the 
Theoretical Statistics and Mathematics Unit of the Indian Statistical Institute, Delhi Centre, for the financial support received 
through an Institute Postdoctoral Fellowship. The third named author would like to thank the Department of Science and 
Technology (DST), India for their support to the Department of Mathematical Science, IISER Mohali  for a financial support in the 
form of the FIST grant (File No. SR/FST/MS-I/2019/46(C)). 

\subsection*{Data Availability Statement} There is no data associated with this manuscript.

\subsection*{Declaration}
The authors have no conflicts of interest to declare that are relevant to the content of this article.

\end{document}